\documentclass[12pt,a4paper]{amsart}

\usepackage{hyperref}
\usepackage{color}

\usepackage[dvipdfm]{graphicx}
\usepackage{pb-diagram,pb-xy}

\usepackage{amsmath,amsfonts,amssymb}

\usepackage{latexsym,mathrsfs,amsthm, amscd, url}
\usepackage{psfrag}
\usepackage[all,graph]{xy}
\usepackage{pb-diagram}
\xyoption{import}
\def\R{\mathbb{R}}
\def\Z{\mathbb{Z}}
\def\N{{\mathbb{N}}}
\def\Q{{\mathbb{Q}}}
\def\C{{\mathbb{C}}}
\def\L{\Lambda}
\def\CP1{\mathbb{CP}^1}

\def\wt{\widetilde}
\def\ol{\overline}
\def\a{\alpha}
\def\b{\beta}
\def\g{\gamma}

\def\la{\lambda}

\def\ac2{\mathcal{AC}_2}

\def\ba{\begin{array}}
\def\ea{\end{array}}
\def\bn{\begin{enumerate}}
\def\en{\end{enumerate}}

\def\ll{\langle}
\def\ti{\tilde}
\def\rr{\rangle}

\def\vph{\varphi}
\def\zt{\Z[t^{\pm 1}]}
\def\toiso{\xrightarrow{\simeq}}
\def\favespace{(\partial E_C^\vph \setminus Y_a^\vph)}
\def\favespacen{\partial E_C^\vph \setminus Y_a^\vph}
\def\fs{\favespace}
\def\fsn{\favespacen}
\def\HH{\mathcal{H}}

\theoremstyle{plain}
\newtheorem*{theorem*}{Theorem}
\newtheorem{theorem}{Theorem}[section]
\newtheorem{proposition}[theorem]{Proposition}
\newtheorem{lemma}[theorem]{Lemma}
\newtheorem{corollary}[theorem]{Corollary}

\newtheorem*{claim}{Claim}

\theoremstyle{definition}

\newtheorem{remark}[theorem]{Remark}

\newtheorem*{fact*}{Fact}

\def\sm{\setminus}
\def\zst{\Z[s^{\pm 1},t^{\pm 1}]}
\DeclareMathOperator{\lk}{lk}
\DeclareMathOperator{\Ext}{Ext}
\DeclareMathOperator{\Hom}{Hom}

\DeclareMathOperator{\im}{im}

\DeclareMathOperator{\cl}{cl}

\DeclareMathOperator{\Aut}{Aut}

\DeclareMathOperator{\pr}{pr}

\DeclareMathOperator{\tr}{tr}

\DeclareMathOperator{\Id}{Id}

\DeclareMathOperator{\GL}{GL}
\DeclareMathOperator{\SL}{SL}
\DeclareMathOperator{\aut}{Aut}

\DeclareMathOperator{\id}{Id}

\begin{document}

\title{Links not concordant to the Hopf link}


\author{Stefan Friedl}
\address{Mathematisches Institut\\ Universit\"at zu K\"oln\\   Germany}
\email{sfriedl@gmail.com}

\author{Mark Powell}
\address{University of Edinburgh, United Kingdom}
\email{M.A.C.Powell@sms.ed.ac.uk}

\def\subjclassname{\textup{2000} Mathematics Subject Classification}
\expandafter\let\csname subjclassname@1991\endcsname=\subjclassname
\expandafter\let\csname subjclassname@2000\endcsname=\subjclassname


\maketitle

\begin{abstract}
We give new Casson--Gordon style obstructions for a two--component link to be topologically concordant to the Hopf link.
\end{abstract}

\section{Introduction}

\subsection{Definitions and previous results}
An
\emph{$m$--component (oriented) link} is an embedded locally flat ordered
collection of $m$ disjoint (oriented) circles in~$S^{3}$.
 We say that two $m$--component oriented
links $L=L_1\,\cup \dots \cup\, L_m$ and $L'=L'_1\,\cup \dots \cup\, L'_m$ are \emph{concordant} if there exists a collection of $m$
disjoint, embedded, locally flat, oriented annuli $A_1,\dots,A_m$ in $S^{3}
\times [0,1]$ such that $\partial A_i= L_i\times \{0\} \cup -L_i'\times
\{1\}, i=1,\dots,m$.

Throughout the paper we denote by $H$ the Hopf link equipped with the orientation such that the linking number is $1$, i.e. $H$ consists of two generic fibres of the Hopf map \[\ba{rcl} h \colon S^3 &\to& S^2 \\ (z_1,z_2) &\mapsto & [z_1,z_2], \ea\]
where $S^3 \subset \mathbb{C}^2$ is the set of $(z_1,z_2) \in \C^2$ with $|z_1|^2 + |z_2|^2 = 1$ and $S^2 \cong \CP1$.

There is a rich literature on obstructions for two links to be concordant. Roughly speaking the obstructions fall into the following categories:
\bn
\item obstructions to the components (viewed as knots) being concordant (the landmark papers being the work of Levine \cite{Levine}, Casson and Gordon \cite{CassonGordon} and Cochran, Orr and Teichner \cite{COT});
\item Milnor's $\ol{\mu}$-invariants (see Milnor \cite{Mil57}), in particular the linking numbers;
\item obstructions from the multivariable Alexander polynomial and the Levine-Tristram signatures (see Murasugi \cite{Mu67}, Kawauchi \cite{Ka77}\cite[Theorem~B]{Ka78}, Nakagawa \cite{Na78} and Tristram \cite{Tristram69});
\item twisted Alexander polynomials, twisted signature invariants and $L^2$--signature invariants, where the corresponding representations factor through a nilpotent quotient (e.g. a $p$--group) or the algebraic closure of the link group
(see  Cha and Ko \cite{CK99,CK06}, Cha and Friedl \cite{cha_twisted_2010}, Friedl \cite{Fri05}, Harvey \cite{Ha08}, Levine \cite{Le94} and Smolinsky \cite{Sm89});
furthermore see  Cha and Orr \cite{chaorr} for a closely related approach;
\item twisted Alexander polynomials and twisted signature invariants of boundary links, where the corresponding representations factor through the canonical epimorphism onto the free group (see Levine \cite{Le07});
\item by considering a prime power branched cover over one component of a link, one can view the preimages of the other components as a link in a rational homology sphere;
using this reduction in the number of components one can reduce problems in link concordance to the study of `rational knot concordance' (see Cochran-Orr \cite{CO90,CO93},
 Cha-Kim \cite{CK08}, Cha-Livingston-Ruberman \cite{CLR08} and Cha-Kim-Ruberman-Strle \cite{CKRS10});
\item by doing a Dehn filling on one component one can also reduce problems in link concordance to the study of knot concordance in integral or rational homology spheres.
This approach, which  for example appears in Cha-Ko \cite{CK99b} and the aforementioned \cite{CKRS10}, is often referred to as `blowing down along one component'.

\en
Davis \cite{Da06}  proved a partial converse to the above obstructions:  if $L$ is a 2-component link with linking number one and such that the multivariable Alexander polynomial $\Delta_L\in \Z[s^{\pm 1},t^{\pm 1}]$ equals one, then the link is concordant to the Hopf link. This result generalised earlier results by Freedman (\cite{Fre84,FQ}) on knot concordance and by Hillman \cite{Hi02}.

In this paper we will give new obstructions to a link being concordant to the Hopf link by using signature invariants corresponding to appropriately chosen metabelian representations.
This approach is inspired by the classic work of Casson and Gordon on knot concordance \cite{CassonGordon}. Generalising their work to the setting of link concordance
requires overcoming several delicate technical issues. One particularly difficult step (Theorem \ref{Thm:main_thm_FP10} in the present paper) appeared in an earlier paper by the authors
\cite{FP10}.

\subsection{The linking pairing for 2--component links with linking number 1}
In the following we let $L$ be a 2--component link of linking number 1.
Throughout
this paper we write $X = X_L=S^3\sm \nu L$, where $\nu L$ is a tubular
neighbourhood of $L$.
The boundary $\partial X_L $ can be written as the union of two tori $\partial X_{L}^1$ and $\partial X_{L}^2$ which correspond to the two link components.
Moreover, for each $i$, we can decompose $\partial X_{L}^i \cong S^1 \times D^1 \cup_{S^1 \times S^0} S^1 \times D^1$, where each cross section $S^1 \times \{pt\}$ is a meridian of the link component $L_i$.  We denote this splitting by \[\partial X_{L}^i \cong Y_{a}^i \cup_{S^1 \times S^0} Y_{b}^i,\]
for $i = 1,2$.  Define $Y_a = Y_{a}^1 \sqcup Y_{a}^2$ and $Y_b = Y_{b}^1 \sqcup Y_{b}^2$.

Note that we have a canonical isomorphism $H_1(X_L;\Z)\toiso \Z^2$ which we will use to identify the groups. Now let $\varphi\colon H_1(X_L;\Z)=\Z^2\to A$ be an epimorphism onto a finite abelian group.
In the following we denote by $p\colon X_L^\varphi\to X$  the corresponding covering space and we write $Y_a^\varphi=p^{-1}(Y_a)$ and $Y_b^\varphi=p^{-1}(Y_b)$.
In Section \ref{section:defn_of_linking_forms} we will see that there exists a non--singular linking form
\[\lambda \colon TH_1(X^\varphi,Y_a^\varphi;\Z)\times TH_1(X^\varphi,Y_{a}^\varphi;\Z)\to \Q/\Z,\]
where, given an abelian group $G$, we write $TG:=\mbox{Tor}_\Z(G)$.

Now suppose that $L$ is concordant to the Hopf link via a concordance $C \cong S^1 \times I \subset S^3 \times I$.  Denote the exterior of the concordance by $E_C$.  The inclusion maps $X_L \to E_C$ and $X_H \to E_C$ are $\Z$-homology equivalences. Note that  $\vph$ canonically extends to a homomorphism $H_1(E_C;\Z)\to A$ which we also denote by $\varphi$.
In the following we denote by $p\colon E_C^\varphi\to E_C$ the corresponding covering space.

Before we state our first result we need to introduce one more definition.
We say that a homomorphism $\varphi\colon \Z^2\to A$ to a finite abelian group is \emph{admissible}
if $\varphi$ is given by the direct sum of two projections on each summand of $\Z^2$, i.e. $\varphi$ has to be of the form
$\Z^2=\Z\oplus \Z\to \Z_k\oplus \Z_l$. 
We now have the following proposition.

\begin{proposition}\label{prop:linkformintro}
Let $L$ be an oriented 2--component link with linking number $1$ which is concordant to the Hopf link via a concordance $C$.  Then for any admissible epimorphism $\varphi\colon \Z^2\to A$ to a group of prime power order,
\[P := \ker(TH_1(X^\vph,Y_a^\vph;\Z) \to TH_1(E_C^{\vph},Y_a^\vph;\Z)),\]
is a metaboliser of $\lambda$, i.e.  $P = P^{\bot}$.
\end{proposition}

Following the philosophy of \cite{CassonGordon}, we will use this proposition to show that certain metabelian representations will extend over a concordance exterior.

\subsection{Definition of the invariant $\tau(L,\chi)$}

We define the closed manifold $$M_L := X_L \cup_{\partial X_L} X_H:$$ an ordering and an orientation of the link components of $L$ and $H$ determines a canonical identification of $\partial X_L$ with $\partial X_H$ which identifies the corresponding meridians and longitudes.  We say that a homomorphism $\varphi \colon H_1(M_L;\Z) \to A$ is admissible if it restricts to an admissible homomorphism on $$H_1(X_L;\Z) \subseteq H_1(M_L;\Z) \cong H_1(X_L;\Z) \oplus \Z.$$  We denote the corresponding finite cover by $M_L^\vph$.  Note that $M_H^\vph$ is a 3-torus.  We denote the 3-torus $S^1 \times S^1 \times S^1$ by $T$.
Let $$\phi' \colon \pi_1(M_L) \to \HH':= H_1(M_L;\Z) \cong \Z^3$$ be the Hurewicz map.  Let $$\phi \colon \pi_1(M_L^\vph) \subseteq \pi_1(M_L) \to \HH'$$ be the restriction of $\phi'$, and define $\HH := \im \phi$.  Note that also $\HH \cong \Z^3$.  We pick an isomorphism $$\psi \colon \pi_1(T) \toiso H_1(T;\Z) \toiso \HH.$$  Finally let $$\chi\colon H_1(M_L^\vph;\Z) \to S^1 \subset \C$$ be a character
which factors through $\Z_{q^l}$ for some prime $q$, and denote the trivial character by $\tr$.

Following Casson-Gordon, in Section \ref{section:deftau}, we define a Witt group obstruction
\[\tau(L,\chi) \in L^0(\C(\HH)) \otimes_{\Z} \Z[1/q],\]
in terms of the twisted intersection form of a 4-manifold whose boundary is a finite number of copies of $(M_L^\vph \sqcup -T, \chi \times \phi \sqcup \tr \times \psi)$.

\subsection{Statement of the main theorem and examples}

The following is  our main theorem.

\begin{theorem}\label{Thm:mainthm_intro}
Suppose that $L$ is concordant to the Hopf link.  Then for any admissible homomorphism $\varphi \colon H_1(M_L) \to A$ to a finite abelian $p$-group $A$, there exists a metaboliser $P = P^\bot$ for the linking form
\[TH_1(X_L^\varphi,Y_a^\varphi)\times TH_1(X_L^\varphi,Y_{a}^\varphi)\to \Q/\Z\]
with the following property: for any character of prime power order $\chi \colon H_1(M_L^\varphi) \to \Z_{q^k} \to S^1$, which satisfies that $\chi|_{H_1(X_L^{\vph})}$ factors through \[\chi|_{H_1(X_L^{\vph})} \colon H_1(X_L^\vph) \to H_1(X_L^\vph,Y_a^\vph) \xrightarrow{\delta} \Z_{q^k} \hookrightarrow S^1\] and that $\delta$ vanishes on $P$, we have that
\[\tau(L,\chi) = 0 \in L^0(\C(\HH)) \otimes_{\Z} \Z[1/q].\]
\end{theorem}

This theorem is reminiscent of `Casson-Gordon type' obstruction theorems in the context of knot concordance, especially of the main theorem of \cite{CassonGordon}.  The proof of Theorem \ref{Thm:mainthm_intro} is indeed guided by the knot theoretic case coupled with the main result of \cite{FP10}.
 In particular we will show that if a link $L$ is concordant to $H$ and if $\chi$ satisfies the conditions of the theorem, then there exists a 4-manifold $W_C^\vph = E_C^\vph \cup (X_H^\vph \times I)$ whose boundary is $M_L^\vph \sqcup -M_H^\vph=M_L^\vph \sqcup -T$, and which has trivial (twisted) intersection form.

In Section \ref{section:examples} we will discuss an example of a link for which our main theorem detects that it is not concordant to the Hopf link.
More precisely, we start out with a 2--component link with trivial components which is concordant to the Hopf link and we then take an appropriate satellite.
We then show that our obstructions detect that the resulting link is not concordant to the Hopf link, but we will see that the methods (1) to (6) listed above do not detect this fact.
It is likely though that our example gets detected by first blowing down and then considering the Casson--Gordon invariants of the resulting knot.

\begin{remark}
\bn
\item When we exhibit our example, we make use of a Maple program to calculate the homology of a finite cover of the link complement explicitly.  We also use Maple to identify a curve in a link diagram, for use in a satellite construction, which lifts to a generator of the torsion part of homology.
\item One theme of this paper is that many arguments which work in the case of knots, for constructing concordance obstructions which involve representations that factor through derived series quotients, can be extended to the case of two component links of linking number one.  Extending these arguments to the case of more components or other linking numbers seems much harder.
\item It is an interesting question whether there also exists a `Cochran--Orr--Teichner type' obstruction to a 2--component link being concordant to a Hopf link.  Added in proof: Jae Choon Cha has recently answered this question in the affirmative: see \cite{Cha_symmetric_Whitney_towers}.
\item We conjecture that Theorem \ref{Thm:mainthm_intro} obstructs $L$ from being ``\emph{height $(3.5)$-Whitney tower/Grope concordant}'' to the Hopf link in $S^3 \times I$ (compare \cite[Sections~8~and~9]{COT}, \cite{CST11}).  Added in proof: Min Hoon Kim informs us that he has confirmed the conjecture: the obstructions of this paper indeed vanish for links which are height $(3.5)$-Whitney tower/Grope concordant to the Hopf link.  At the time of writing his paper is in preparation.
\item Blowing down is in practice a remarkably effective method for studying link concordance. Even though we do not think that one can recover Theorem \ref{Thm:mainthm_intro} using that approach, it seems that it will be difficult to find an example of a link that Theorem \ref{Thm:mainthm_intro} can show is not concordant to the Hopf link but for which the blowing down approach fails.
\item Whereas blowing down is rather effective as a way of obstructing links from being concordant to the Hopf link, it cannot obstruct the existence of height $(3.5)$ Whitney tower or Grope concordances.  From this point of view the conjecture in (4) would show that the obstructions of this paper are stronger than blowing down.
\en
\end{remark}

\subsection{Organization of the paper}

The paper is organised as follows.  In Section \ref{section:finite_covers_linking_forms} we define linking forms on the torsion part of the first homology of finite covers of the link exterior, and use them to control the representations which extend over the exterior of a concordance.  In Section \ref{section:obstructions} we define the obstruction $\tau(L,\chi)$ and prove Theorem \ref{Thm:mainthm_intro}.  Finally, in Section \ref{section:examples} we construct examples of links which are not concordant to the Hopf link, and discuss the vanishing of many previously known obstructions.
We conclude with a short appendix in which we prove a formula relating the homology of finite abelian covers of a 2-component link with linking number one to values of the multivariable Alexander polynomial.

\section{Finite Covers and Linking Forms}\label{section:finite_covers_linking_forms}

\subsection{Injectivity Result}\label{section:injectivity_result}

Below we will quote a theorem from the authors' previous paper \cite{FP10}.  We will make use of this at several points in the construction of the obstruction theory.  The special case that $\mathcal{H} \cong \Z$ was essentially proved by Letsche in \cite{Let00}.  Our paper makes use of the case that $\mathcal{H} \cong \Z^3$.  Given a prime $q$, recall that a $q$-group is a finite group of $q$-power order.  Let $\pi$ be a group, let $Q$ be a field and let $\HH$ be a free abelian group.  Representations $\phi \colon \pi \to \HH$ and $\a \colon \pi \to \GL(k,Q)$ give rise to a right $\Z[\pi]$-module structure on $Q^k\otimes_{Q} Q[\HH] = Q[\HH]^k$ as follows:
\[ \ba{rcl} Q^k\otimes_{Q} Q[\HH] \,\,\times \,\Z[\pi]&\to & Q^k\otimes_{Q} Q[\HH]  \\
(v\otimes p,g)&\mapsto & (v\cdot \a(g)\otimes p\cdot \phi(g)).\ea \]

\begin{theorem}\label{Thm:main_thm_FP10}
Let $\pi$ be a group, let $q$ be a prime and let $f \colon M \to N$ be a morphism of projective left $\Z[\pi]$-modules such that
\[\Id \otimes f\colon \Z_q \otimes_{\Z[\pi]} M \to \Z_q \otimes_{\Z[\pi]} N\]
is injective.  Let $\phi \colon \pi \to \mathcal{H}$ be a homomorphism to a torsion-free abelian group and let $\a \colon \pi \to \GL(k,Q)$ be a representation, with $Q$ a field of characteristic zero.  If $\a|_{\ker (\phi)}$ factors through a $q$-group, then
\[\Id\otimes f\colon Q[\HH]^k\otimes_{\Z[\pi]} M \to Q[\HH]^k\otimes_{\Z[\pi]} N\]
is also injective.
\end{theorem}

A special case of this theorem applies to prove the following lemma, which allows us to control the rank of the $\Q$-homology of finite covers of the link exterior.

\begin{lemma}\label{lemma:Zhomequiv_implies_torsion}
Suppose that $f \colon X \to E$ is a map of finite CW complexes which induces a $\Z$-homology equivalence
\[f_* \colon H_*(X;\Z) \xrightarrow{\simeq} H_*(E;\Z).\]
Let $\varphi \colon \pi_1(E) \to A$ be an epimorphism to a finite abelian $p$-group $A$.  Let $E^\vph$ be the induced cover and let $X^\vph$ be the corresponding pull-back cover.  Then
\[f_* \colon H_*(X^\vph;\Q) \xrightarrow{\simeq} H_*(E^\vph;\Q).\]
is an isomorphism.
\end{lemma}
\begin{proof}
We refer to \cite[Proposition~2.10]{COT} for the details of a standard chain homotopy lifting argument, which should be applied to the relative chain complex
\[C_* := C_*(E,X;\Q[A]).\]
This argument was also given in the proof of \cite[Proposition~4.1]{FP10}.  We insert into the argument, at the point where injectivity is required, an application of Theorem \ref{Thm:main_thm_FP10}. We use the special case that $\HH \cong \{0\}$, $q=p$, $Q = \Q$, and $\a \colon \pi_1(E) \to A \to GL(k,\Q)$ is given by $k = |A|$ and the regular representation $A \to \Aut_Q(\Q[A])$.  This special case was of course well-known before our contribution.
\end{proof}

\subsection{Definition of linking forms}\label{section:defn_of_linking_forms}

Let $L$ be an oriented 2-component link with linking number $1$. We write $X=X_L$.  Let $\varphi\colon \Z^2\to A$ be any epimorphism onto a finite abelian group.
We will now define a non-singular linking form:
\[\lambda=\lambda_L \colon TH_1(X^\varphi,Y_a^\varphi;\Z)\times TH_1(X^\varphi,Y_{a}^\varphi;\Z)\to \Q/\Z.\]
 First, we have the Poincar\'{e} duality isomorphism.
\[TH_1(X^\varphi,Y_a^\varphi;\Z) \xrightarrow{\simeq} TH^2(X^\varphi,Y_b^\varphi;\Z).\]
By the universal coefficient theorem, since torsion in the cohomology $H^2(X^\varphi,Y_b^\varphi;\Z)$ maps to zero in $\Hom_{\Z}(H_2(X^\varphi,Y_b^\varphi;\Z),\Z)$, there is a canonical isomorphism:
\[TH^2(X^\varphi,Y_b^\varphi;\Z) \xrightarrow{\simeq} \Ext_{\Z}^1(TH_1(X^\varphi,Y_b^\varphi;\Z),\Z).\]
For the reader's convenience we now recall the proof of the following well--known claim:

\begin{claim}
Given any finite abelian group $G$ there exists a canonical isomorphism
\[ \Ext_{\Z}^1(G,\Z)\xrightarrow{\simeq} \Hom(G,\Q/\Z).\]
\end{claim}

Let $G$ be any finite abelian group.
The short exact sequence of coefficients \[0 \to \Z \to \Q \to \Q/\Z \to 0\] gives rise to a long exact sequence of cohomology
\[ \Ext^0_\Z(G,\Q)  \to  \Ext^0_\Z(G,\Q/\Z) \to \Ext^1_\Z(G,\Z) \to \Ext^1_\Z(G,\Q).\]
It is a simple calculation from the definition to see that $\Ext^*_\Z(G,\Q)=0$; we thus obtain an isomorphism $ \Ext^0_\Z(G,\Q/\Z)  \xrightarrow{\simeq} \Ext^1_\Z(G,\Z)$.
Combining the inverse of this isomorphism with the canonical isomorphism
$ \Ext_{\Z}^0(G,\Q/\Z)\xrightarrow{\simeq}  \Hom(G,\Q/\Z)$ shows that we have a canonical isomorphism $\Ext_{\Z}^1(G,\Z)\xrightarrow{\simeq} \Hom(G,\Q/\Z)$.
This concludes the proof of the claim.

\begin{claim}
There exists a canonical isomorphism of relative homology groups:
\[TH_1(X^\varphi,Y_a^\varphi;\Z) \xrightarrow{\simeq} TH_1(X^\varphi,Y_b^\varphi;\Z).\]
\end{claim}

First note that the two inclusion maps \[\bigsqcup_2 \, S^1 \times D^1 \cong Y_a \hookrightarrow X_L\] and \[\bigsqcup_2 \, S^1 \times D^1 \cong Y_b \hookrightarrow X_L\] are isotopic. We thus get an  isomorphism of relative homology groups:
\[TH_1(X^\varphi,Y_a^\varphi;\Z) \xrightarrow{\simeq} TH_1(X^\varphi,Y_b^\varphi;\Z).\]
To see that this isomorphism is canonical, consider the following $\Z$-coefficient commutative diagram.
\[\xymatrix{H_1(Y_a^\vph) \ar[r] \ar[d]^{r}_{\cong} & H_1(X_L^\vph) \ar[r] \ar[d]^f_{\cong} & H_1(X_L^\vph,Y_a^\vph) \ar[r] \ar[d] & H_0(Y_a^\vph) \ar[d]_{\cong}^r \ar[r] & H_0(X_L^\vph) \ar[d]^f_{\cong}\\
H_1(Y_b^\vph) \ar[r] & H_1(X_L^\vph) \ar[r] & H_1(X_L^\vph,Y_b^\vph) \ar[r] & H_0(Y_b^\vph) \ar[r] & H_0(X_L^\vph)}\]
The maps labelled $r$ are induced by a rotation of the two boundary tori which sends $Y_a^\vph$ to $Y_b^\vph$.  It is a homeomorphism so induces isomorphisms on homology.  The maps labelled $f$ are induced by an isotopy of homeomorphisms between this rotation of $\partial X_L^\vph$ and the identity map of $\partial X_L^\vph$, \[i \colon \partial X_L^\vph \times I \to \partial X_L^\vph.\]  We map a neighbourhood $\partial X_L^\vph \times I$ of the boundary to itself using \[\ba{rcl} f \colon \partial X_L^\vph \times I &\to& \partial X_L^\vph \times I\\ (x,t) & \mapsto & (i(x,t),t).\ea\]  Extending $f$ by the identity to the rest of $X_L^\vph$ we obtain a homeomorphism, which induces the isomorphisms $f$ on homology.  The map of relative groups is also induced by the map $f$: note that $f$ maps $Y_a^\vph$ to $Y_b^\vph$.  By the 5-lemma, the map of relative groups is an isomorphism.  Since $H_0(Y_a^\vph)$ is torsion-free, an element in $TH_1(X_L^\vph,Y_a^\vph)$ maps to zero in $
 H_0(Y_a^\vph)$, and so lies in the image of $H_1(X_L^\vph)$.  This of course also holds with $a$ replaced by $b$.  Any element of $H_1(X_L^\vph)$ can be taken to be disjoint from a collar neighbourhood of the boundary, which implies that the map $TH_1(X^\varphi,Y_a^\varphi) \xrightarrow{\simeq} TH_1(X^\varphi,Y_b^\varphi)$ is canonical: it does not depend on choices of $r$ or $i$, since their effect is restricted to a neighbourhood of the boundary of $X_L^\vph$. This concludes the proof of the claim.

Combining the above isomorphisms yields
\[TH_1(X^\varphi,Y_a^\varphi;\Z) \xrightarrow{\simeq}\Hom_{\Z}(TH_1(X^\varphi,Y_b^\varphi;\Z),\Q/\Z) \]\[\toiso \Hom_{\Z}(TH_1(X^\varphi,Y_a^\varphi;\Z),\Q/\Z),\]
which
defines the linking form $\lambda=\lambda_L$ and shows that it is non-singular.  Standard arguments show that this linking form is symmetric.

\subsection{Linking forms and concordance}

\begin{proposition}\label{prop:linkform}
Let $L$ be an oriented 2-component link which is concordant to the Hopf link via a concordance $C$. We write $X=X_L$.  Let  $\varphi\colon \Z^2\to A$ be an
 admissible epimorphism onto a group of prime power order.
Then
\[P := \ker(TH_1(X^\vph,Y_a^\vph;\Z) \to TH_1(E_C^{\vph},Y_a^\vph;\Z))\]
is a metaboliser of the linking form  $\lambda_L$.
\end{proposition}

Given an abelian group $G$ we write  $G^{\wedge} := \Hom_{\Z}(G,\Q/\Z)$. Note that $\lambda_L$ defines
a map $TH_1(X^\vph,Y_a^\vph;\Z)\to TH_1(X^\vph,Y_a^\vph;\Z)^\wedge$.
The main technical ingredient in the proof of Proposition \ref{prop:linkform} is the following lemma:

\begin{lemma}\label{lem:forms}
There are isomorphisms
\[\vartheta_1 \colon TH_1(E_C^\vph,Y_a^\vph;\Z) \to  TH_2(E_C^\vph,X^\vph;\Z)^\wedge \]
and
\[\vartheta_2 \colon TH_2(E_C^\vph,X^\vph;\Z) \to  TH_1(E_C^\vph,Y_a^\vph;\Z)^\wedge \]
such that the following diagram commutes:
\[\xymatrix{ TH_2(E_C^\vph,X^\vph) \ar[r]^{\partial} \ar[d]^{\vartheta_2}_{\cong} & TH_1(X^\vph,Y_a^\vph) \ar[r]^{i_*} \ar[d]^{\lambda}_{\cong} & TH_1(E_C^\vph,Y_a^\vph) \ar[d]^{\vartheta_1}_{\cong} \\
TH_1(E_C^\vph,Y_a^\vph)^{\wedge} \ar[r]^{i^{\wedge}}  & TH_1(X^\vph,Y_a^\vph)^{\wedge} \ar[r]^{\partial^{\wedge}} &  TH_2(E_C^\vph,X^\vph)^{\wedge}.
}\]
\end{lemma}

A standard argument shows how to deduce Proposition \ref{prop:linkform} from Lemma \ref{lem:forms}. We now provide the argument for the reader's convenience.

\begin{proof}[Proof of Proposition \ref{prop:linkform}]
First note that the long exact sequence in homology of the pair $(E_C^\vph,X^\vph)$ restricts to an exact sequence
\[TH_2(E_C^\vph,X^\vph) \to TH_1(X^\vph,Y_a^\vph) \to TH_1(E_C^\vph,Y_a^\vph)\]
 since $TH_2(E_C^\vph,X^\vph) = H_2(E_C^\vph,X^\vph)$ by Lemma \ref{lemma:Zhomequiv_implies_torsion}.  This implies that the rows in
the above commutative diagram are exact.
Defining
\[P := \ker(i_*\colon TH_1(X^\vph,Y_a^\vph) \to TH_1(E_C^{\vph},Y_a^\vph)),\]
as in the statement of the proposition, we will now show that $P = P^{\bot}$.

First, let $x\in P$.  By exactness there exists a $w \in TH_2(E_C^\vph,X^\vph)$ such that $\partial (w) = x$.  Therefore
\[\lambda(x) = \lambda\circ\partial (w) = i^{\wedge} \circ \vartheta_2(w) = \vartheta_2(w)\circ i_*,\]
with the second equality by commutativity of the left square.
If also $y \in P$, \[\lambda(x)(y) = \vartheta_2(w)(i_*(y)) = \vartheta_2(w)(0) = 0,\]
so that $x \in P^{\bot}$ and $P \subseteq P^{\bot}$.

Now let $y \in P^{\bot}$.   This implies
that $\partial^\wedge \circ \la(y)=0$. Since the lower sequence is also exact it follows
 that $\la(y)=i^\wedge(x)$ for some $x\in TH_1(E_C^\varphi,Y_a^\varphi)^\wedge$.
 Since the left square of the diagram commutes, we see that
 $y=\partial(\vartheta_2^{-1}(x))$ i.e. $y\in P$.
 This shows that  $P^{\bot} \subseteq P$.
\end{proof}

The proof of Lemma \ref{lem:forms} will require the remainder of this section, in which for brevity we omit the $\Z$ coefficients for the homology groups.

To define $\vartheta_1$ and $\vartheta_2$, we first use the analogue of the isomorphisms used to define $\lambda.$  For $\vartheta_1$ we have:
\[\ba{rcl}  TH_1(E_C^\vph,Y_a^\vph) &\xrightarrow{\simeq}& TH^3(E_C^\vph, \partial E_C^{\vph} \setminus Y_a^\vph) \xrightarrow{\simeq} \\ \Ext_\Z^1(TH_2(E_C^\vph,\partial E_C^{\vph} \setminus Y_a^\vph),\Z) &\xrightarrow{\simeq}& \Ext_\Z^0(TH_2(E_C^\vph,\partial E_C^{\vph} \setminus Y_a^\vph),\Q/\Z)\\
&\xrightarrow{\simeq}& \Hom(TH_2(E_C^\vph,\partial E_C^{\vph} \setminus Y_a^\vph),\Q/\Z). \ea \]
Similarly, for $\vartheta_2$, we have
\[\ba{rcl} \vartheta_2 \colon TH_2(E_C^\vph,X^\vph) &\xrightarrow{\simeq}& TH^2(E_C^\vph,X_H^\vph) \xrightarrow{\simeq} \\ \Ext_\Z^1(TH_1(E_C^\vph,X_H^\vph),\Z) &\xrightarrow{\simeq}& \Ext_\Z^0(TH_1(E_C^\vph,X_H^\vph),\Q/\Z)\\
 &\xrightarrow{\simeq}& \Hom(TH_1(E_C^\vph,X_H^\vph),\Q/\Z). \ea \]

To finish the construction of the non-singular relative linking pairings, we will prove the following two lemmas:

\begin{lemma} \label{lem:isom1}
The inclusion $X^\vph\to  \partial E_C^\vph \setminus Y_a^{\vph}$ induces an isomorphism:
\[ TH_2(E_C^\vph, X^\vph) \xrightarrow{\simeq} TH_2(E_C^\vph, \partial E_C^\vph \setminus Y_a^{\vph}).\]
\end{lemma}

\begin{lemma} \label{lem:isom2}
The inclusion $Y_a^\vph\to X_H^\vph$ induces an isomorphism:
\[ TH_1(E_C^\vph, Y_a^{\vph}) \xrightarrow{\simeq} TH_1(E_C^\vph, X_H^\vph).\]
\end{lemma}

Assuming Lemmas \ref{lem:isom1} and \ref{lem:isom2} we use these maps to finish the construction of $\vartheta_1$ and $\vartheta_2$.
It is straightforward to see that the forms $\vartheta_1,\vartheta_2$ and $\lambda$ fit into a commutative diagram as postulated in Lemma \ref{lem:forms}.  It thus remains to prove
Lemmas \ref{lem:isom1} and \ref{lem:isom2}.\\

We continue with the notation as above.
We denote the two components of $Y_b$ by $Y_{b1}$ and $Y_{b2}$.
For $j=1,2$ we pick one component $\ti{Y}_{bj}$ of the preimage of $Y_{bj}$ under the
covering map $Y_b^\varphi\to Y_b$. We write $\ti{Y}_b=\ti{Y}_{b1}\cup \ti{Y}_{b2}$.
Note that each component of $Y_b^\vph$ is isotopic in $\partial X$ to one of the two components of $\ti{Y}_b$. We thus obtain a canonical map $Y_b^\vph\to \ti{Y}_b$.

We now have the  following lemma:

\begin{lemma} \label{lem:diagram}
Let $\varphi$ be admissible. Consider the following diagram
\[ \xymatrix{
&H_1(E_C^\varphi)& \\
H_1(X^\varphi)\ar[ur] & H_1(Y_b^\varphi)\ar[d]\ar[l]\ar[r]\ar[u] & H_1(X_H^\varphi) \ar[ul] \\ & H_1(\ti{Y}_b).\ar[ul]\ar[ur]&}\]
Then the following hold:
\bn
\item the diagram commutes,
\item the bottom vertical map is surjective,
\item the bottom right diagonal map is an isomorphism,
\item all top diagonal maps are isomorphisms with rational coefficients,
\item the top vertical map is rationally surjective,
\item the bottom left diagonal map is injective and an isomorphism over $\Q$,
\item the bottom left diagonal map splits.
\en
\end{lemma}

\begin{proof}[Proof of Lemma \ref{lem:diagram}]
It is evident that the diagram commutes and that the bottom vertical map is surjective (this is (1) and (2)).
Now recall that $X_H$ is homeomorphic to a torus times an interval $S^1 \times S^1 \times I$. We denote the two meridians of $H$ by $x$ and $y$. Recall that we identified $H_1(X_H)$ with $\Z^2$ such that $x$ corresponds to $(1,0)$ and $y$ corresponds to $(0,1)$.
Let $\varphi\colon H_1(X_H) \to A$ be any epimorphism onto a finite abelian group. We denote by $m$ and $n$ the orders of $\varphi(x)\in A$ and $\varphi(y)\in A$.
Lifting to the covering space induced by $\varphi$, and then considering $H_1(X_H^\vph)$ as a subgroup of $H_1(X_H)$, we see that
\[ \im(H_1(\ti{Y}_b)\to H_1(X_H^\varphi) \to H_1(X_H))=\Z x^m \oplus \Z y^n.\]
On the other hand we have $\im(H_1(X_H^\varphi) \to H_1(X_H)) = \ker(\vph \colon H_1(X_H) = \Z^2\to A)$.
It is straightforward to verify that $\im(H_1(\ti{Y}_b)\to H_1(X_H^\varphi))=H_1(X_H^\varphi)$ if and only if
\[ \ker(\varphi\colon \Z^2\to A)=\ker(\varphi\colon \ll x\rr \to A)\oplus \ker(\varphi\colon \ll y\rr \to A).\]
The latter condition in turn is satisfied if and only if $\varphi$ is admissible.
This shows that the bottom right diagonal map is an isomorphism, which is (3).

Note that by Lemma \ref{lemma:Zhomequiv_implies_torsion} the top diagonal maps are isomorphisms with rational coefficients, which is (4).  Since the bottom right diagonal map is an isomorphism over the integers, it follows by commutativity, going both ways around the diamond, that the bottom left diagonal map is also an isomorphism over the rationals, which is part of (6).  It also follows by commutativity, (2), (3) and (4), that the top vertical map is surjective over the rationals, which is (5).  Note note that  $H_1(\ti{Y}_b)$ is torsion-free.  It follows that the bottom left diagonal map is injective, which shows (6).

Finally, consider the following
 commutative diagram of exact sequences
\[ \xymatrix{0\ar[r]& H_1(\widetilde{Y}_b^\varphi)\ar[r]\ar[d] & H_1(Y_b)\ar[r]\ar[d]^= & A\ar[d]^=\ar[r]&0\\
  & H_1(X^\varphi)\ar[r] & H_1(X)\ar[r] & A\ar[r]&0.}\]
Note that the right two vertical maps are isomorphisms. It is a straightforward diagram chase to see that the left hand map splits, which is (7).
\end{proof}

For future reference we record the following corollary to Lemma \ref{lem:diagram}.

\begin{corollary}\label{cor:h1yasplits}
Let $\varphi$ be admissible. Then we have an isomorphism
\[ TH_1(X^\varphi,Y_b^\varphi)\cong TH_1(X^\varphi).\]
\end{corollary}

\begin{proof}
By Lemma \ref{lem:diagram} (7) the inclusion induced map $H_1(\widetilde{Y}_b^\varphi)\to H_1(X^\varphi)$ admits a left inverse.  It follows that
\[  TH_1(X^\varphi,\widetilde{Y}_b^\varphi)\cong TH_1(X^\varphi).\]
On the other hand it follows easily from comparing the long exact sequences of the pairs $(X^\varphi,Y_b^\varphi)$ and $(X^\varphi,\widetilde{Y}_b^\varphi)$ that
\[  TH_1(X^\varphi,\widetilde{Y}_b^\varphi)\cong TH_1(X^\varphi,{Y}_b^\varphi).\]
\end{proof}

\begin{proof}[Proof of Lemma \ref{lem:isom1}]
The map which we claim is an isomorphism arises from restricting to the torsion part in the long exact sequence of the triple $(E_C^\vph, \favespacen, X^\vph\setminus Y_a^\vph \simeq X^\vph)$:
\[\ba{cccccccccc} &&H_2(\favespacen,X^\vph) &\to & H_2(E_C^\vph, X^{\vph})& \to& H_2(E_C^\vph, \favespacen)\\
 &\to& H_1(\favespacen,X^\vph). &\ea\]
By Lemma \ref{lemma:Zhomequiv_implies_torsion}, we have that
\[H_2(E_C^\vph, X^\vph) \cong TH_2(E_C^\vph, X^\vph).\]
It therefore suffices to show  that
\[TH_1(\favespacen,X^\vph) \cong 0\]
and that
\[ H_2(\favespacen,X^\vph) \to H_2(E_C^\vph, X^{\vph})\]
is the zero map.

We first consider the following  Mayer-Vietoris sequence to investigate the homology of $\partial E_C^\vph \setminus Y_a^\vph$, decomposing it as $X^\vph \cup_{Y_b^\vph} X_H^{\varphi}$.
\[
H_1(Y_b^\vph) \to  H_1(X^\vph) \oplus H_1(X_H^\vph) \to  H_1\fs \to
H_0(Y_b^\vph)\to \dots \]
We write $I := \im\big(H_1(X^\vph) \to H_1\fs \big)$.
By Lemma \ref{lem:diagram} (1),(2) and (3), the map $H_1(Y_b^\vph) \to H_1(X_H^\vph)$  is surjective. It then follows from the above exact sequence
that
\[I=\im\big( H_1(X^\vph) \oplus H_1(X_H^\vph) \to  H_1\fs\big).\]
In particular, we get an inclusion
\[ 0 \to  H_1\fs/I \to H_0(Y_b^\vph).\]
Since $H_0(Y_b^\vph)$ is torsion--free, it follows that $H_1\fs/I$ is also torsion-free.
On the other hand the pair $(\partial E_C^\vph\sm Y_a^\vph,X^\varphi)$ gives rise to the following exact sequence
\[ 0 \to H_1\fs/I \to H_1(\fsn,X^\vph) \to H_0(X^\vph).\]
Since any subgroup of $H_0(X^\vph)$ is torsion--free it now follows that
\\$H_1(\fsn,X^\vph)$ is torsion-free.
Put differently,
 \[TH_1(\fsn,X^{\vph}) \cong 0,\] as required.

In order to complete the proof of the lemma, we still have to show that
\[H_2(\favespacen,X^\vph) \to H_2(E_C^\vph, X^{\vph})\]
is the zero map.
The following commutative diagram will guide us through the remainder of the proof:
\[\xymatrix{
0\ar[d] \ar[r] &  H_2((\partial X^1)^\vph) \oplus H_2((\partial X^2)^\vph) \ar[d]\\
H_2(X_H^\vph) \oplus H_2(X^\vph) \ar@{=}[r] \ar[d] & H_2(X_H^\vph) \oplus H_2(X^\vph) \ar[d]\\
 H_2\fs \ar[r] \ar[d]^\partial & H_2(\partial E_C^\vph) \ar[d]\ar[r]& H_2(\partial E_C^\varphi,X^\varphi)\\
 H_1(Y_b^\vph) \ar[r]^-{g} \ar[d]^-{f} & H_1((\partial X^1)^\vph) \oplus H_1((\partial X^2)^\vph) \ar[d]\\
H_1(X_H^\vph) \oplus H_1(X^\vph) \ar[d] \ar@{=}[r] & H_1(X_H^\vph) \oplus H_1(X^\vph)\\
H_1(\fs)& }\]
Note that the two vertical sequences are exact.
We start out with the following claim:

\begin{claim}
The map
\[H_2\fs \to H_2(\fsn,X^\vph)\]
is surjective.
\end{claim}

Note that by Lemma \ref{lem:diagram} (1),(3) and (6) we know that
\[\ker(H_1(Y_b^\vph) \to H_1(X_H^\vph)) = \ker(H_1(Y_b^\vph)\to H_1(\ti{Y}_b)) =\]\[\ker(H_1(Y_b^\vph) \to H_1(X^\vph)). \]
Using the left vertical Mayer-Vietoris sequence in the above diagram and the fact that the map \[H_1(Y_b^\vph) \to H_1(X_H^\vph)\] is surjective, by Lemma \ref{lem:diagram} (1),(2) and (3), this implies that the inclusion induced map
\[ H_1(X^\varphi)\to H_1\fs \]
is injective. The claim follows from considering the long exact sequence in homology of the pair $(\fsn,X^\varphi)$.

By the claim, we just have to show that the composition
\[H_2\fs \to H_2(\fsn,X^\vph) \to H_2(E_C^\vph, X^\vph) \]
is the zero map.  Moreover, the map $H_2(\fsn,X^\vph) \to H_2(E_C^\vph, X^\vph)$ factors as
\[H_2(\fsn,X^\vph) \to H_2(\partial E_C^\vph,X^\vph) \to H_2(E_C^\vph, X^\vph),\]
so it suffices to show that the map
\[H_2(\fsn) \to H_2(\partial E_C^\vph,X^\vph)\]
is the zero map.

First note that the maps $H_i(\partial (X^j)^\varphi)\to H_i(X_H^\varphi), j=1,2$
in the above commutative diagram are induced by homotopy equivalences; in particular they are isomorphisms.
It follows from the second vertical exact sequence that
\[ \im\big(H_2(X_H^\varphi)\to H_2(\partial E_C^\varphi)\big)\subset  \im\big(H_2(X^\varphi)\to H_2(\partial E_C^\varphi)\big).\]
The fact that $H_2(X^\vph)$ maps to zero in $H_2(\partial E_C^\varphi,X^\varphi)$ implies that the map
\[H_2(X_H^\vph) \oplus H_2(X^\vph) \to H_2(\partial E_C^\vph)\to  H_2(\partial E_C^\varphi,X^\varphi)\]
is the zero map.

%

Recall that there are maps labelled as $f$ and $g$ in the preceding large commutative diagram.

\begin{claim}
\[\ker(f) \subseteq \ker(g).\]
\end{claim}

Suppose that $x \in \ker(f)$. Then in particular $x \in \ker(H_1(Y_b^\vph) \to H_1(X_H^\vph))$.  But then, since $X_H^\vph$ is just a product $S^1 \times S^1 \times I$ and $\partial X_H^\vph \cong \partial X^\vph$, the inclusion induced maps $H_1((\partial X^1)^\vph) \to H_1(X_H^\vph)$ and $H_1((\partial X^2)^\vph) \to H_1(X_H^\vph)$ are isomorphisms.  So indeed, using the inverses of these maps and commutativity of the bottom square, we have $x \in \ker(g)$. This concludes the proof of the claim.

A short diagram chase proves that the map $$H_2\fs \to H_2(\fsn,X^\vph) \to H_2(\partial E_C^\vph, X^\vph)$$ is the zero map.  Indeed, let $y \in H_2\fs$.  Then $\partial(y)\in H_1(Y_b^\varphi)$ lies in the kernel of $f$.  Since $\ker(f) \subseteq \ker(g)$, and by the commutativity of the diagram
it follows that the image of $y$ in $H_2(\partial E_C^\varphi)$ lies in the image of $H_2(X_H^\vph) \oplus H_2(X^\vph)$.
But we saw above that the map to $H_2(\partial E_C^\varphi,X^\varphi)$ vanishes on the image of $H_2(X_H^\vph) \oplus H_2(X^\vph)$.

\end{proof}

For the proof of Lemma \ref{lem:isom2} we will need the following lemma,
which we will also use in Section \ref{section:extend}.

\begin{lemma} \label{lem:isomsplit}
Let $J=H$ or $J=L$. We write
\[ F:=H_1(X_J^\vph, Y_a^{\vph};\Z)/\mbox{torsion}.\]
Then there exists a commutative diagram
\[ \xymatrix{TH_1(X_J^\varphi,Y_a^\varphi)\oplus F \ar[d]^{i_*\oplus \id}  \ar[r]^-\cong &   H_1(X_J^\varphi,Y_a^\varphi)\ar[d]^{i_*}\\
TH_1(E_C^\varphi,Y_a^\varphi)\oplus F\ar[r]^-\cong & H_1(E_C^\varphi,Y_a^\varphi).}\]
\end{lemma}

\begin{proof}
We have the following commutative diagram.
\[\xymatrix{
H_1(Y_a^\vph) \ar[r] \ar@{=}[d] & H_1(X_J^\vph) \ar[d] \ar[r] & H_1(X_J^\vph,Y_a^\vph) \ar[r] \ar[d] & H_0(Y_a^\vph) \ar[r] \ar@{=}[d] & H_0(X_J^{\vph}) \ar[d]_{\cong} \\
H_1(Y_a^\vph) \ar[r]  & H_1(E_C^\vph) \ar[r] & H_1(E_C^\vph,Y_a^\vph) \ar[r] & H_0(Y_a^\vph) \ar[r] & H_0(E_C^{\vph})
}\]
It follows from Lemma \ref{lem:diagram} that the two left horizontal maps are rationally surjective.
We write $F':=\ker(H_0(Y_a^\varphi)\to H_0(X_J^\varphi))$, which equals
$\ker(H_0(Y_a^\varphi)\to H_0(E_c^\varphi))$. Note that $F'$, as a subgroup of a torsion-free group is itself torsion-free.
We now obtain the following commutative diagram of short exact sequences:
\[\xymatrix{
0 \ar[r] & H_1(X_J^\vph)/\im(H_1(Y_a^\vph)) \ar[d] \ar[r] & H_1(X_J^\vph,Y_a^\vph) \ar[r] \ar[d] & F'\ar[r]\ar@{=}[d]&0\\
0\ar[r]& H_1(E_C^\vph)/\im(H_1(Y_a^\vph)) \ar[r] & H_1(E_C^\vph,Y_a^\vph) \ar[r] & F'\ar[r]&0.}\]
By the above the two groups on the left are torsion. Since $F'$ is torsion-free it follows that the images of the groups on the left are precisely
the torsion subgroups. We pick an isomorphism $F\to F'$ and a splitting $F'\to H_1(X_J^\vph,Y_a^\vph)$.
Together with the induced splitting $F'\to H_1(X_J^\vph,Y_a^\vph)\to H_1(E_C^\varphi,Y_a^\varphi)$ we obtain the desired commutative diagram.
\end{proof}

\begin{proof}[Proof of Lemma \ref{lem:isom2}]
Recall the statement of Lemma \ref{lem:isom2}: the inclusion $Y_a^\vph\to X_H^\vph$ induces an isomorphism
\[ TH_1(E_C^\vph, Y_a^{\vph}) \xrightarrow{\simeq} TH_1(E_C^\vph, X_H^\vph).\]
We apply Lemma \ref{lem:isomsplit} to the case $J=H$. We obtain the following commutative diagram, where the right hand column contains a portion of the long exact sequence of the triple.
\[ \xymatrix{TH_1(X_H^\varphi,Y_a^\varphi)\oplus F \ar[d]^{i_*\oplus \id}  \ar[r]^-\cong &   H_1(X_H^\varphi,Y_a^\varphi)\ar[d]^{i_*}\\
TH_1(E_C^\varphi,Y_a^\varphi)\oplus F\ar[r]^-\cong & H_1(E_C^\varphi,Y_a^\varphi)\ar[d]\\
& H_1(E_C^\varphi,X_H^\varphi) = TH_1(E_C^\varphi,X_H^\varphi)\ar[d] \\& H_0(X_H^\vph,Y_a^\vph) = 0 }\]
We know that \[H_1(E_C^\varphi,X_H^\varphi) = TH_1(E_C^\varphi,X_H^\varphi)\] by Lemma \ref{lemma:Zhomequiv_implies_torsion}.  We immediately obtain the desired conclusion from the observation that $TH_1(X_H^\varphi,Y_a^\varphi)=0$.  This is a consequence of the proof of Lemma \ref{lem:isomsplit}, which shows that $TH_1(X_H^\vph,Y_a^\vph)$ is precisely $H_1(X_H^\vph)/\im(H_1(Y_a^\vph))$.  But by Lemma \ref{lem:diagram} (1), (2) and (3), $H_1(Y_a^\vph) \to H_1(X_H^\vph)$ is surjective.
\end{proof}

\subsection{Extending characters} \label{section:extend}

\begin{proposition}\label{prop:extend}
Let $L$ be an oriented 2-component link which is concordant to the Hopf link via a concordance $C$. We write $X=X_L$.  Let $\varphi: H_1(X;\Z) = \Z^2 \to A$ be an admissible epimorphism onto a group of prime power order.
Then there exists a metaboliser $P\subset TH_1(X^\vph,Y_a^\vph;\Z)$ for the linking form $\lambda_L$  such that for any prime $q$ and any homomorphism
$\chi\colon H_1(X^\varphi;\Z)\to \Z_{q^k}$
which factors through $H_1(X^\varphi;\Z)\to H_1(X^\varphi,Y_a^\varphi;\Z) \to \Z_{q^k}$
and which vanishes on $P$, there exists an integer $l \geqslant k$ and a homomorphism $H_1(E_C^\varphi;\Z)\to \Z_{q^l}$, such that following diagram commutes:
\[ \xymatrix{ H_1(X^\varphi;\Z)\ar[d] \ar[r]^-{\chi} & \Z_{q^k} \ar[d]^{\cdot q^{l-k}} \\  H_1(E_C^\varphi;\Z)\ar[r]& \Z_{q^l}.}\]
\end{proposition}

\begin{proof}
We start out with the following elementary observation: let $A\to B$ be a homomorphism between finite abelian groups and let $\chi\colon A\to \Z_{q^k}$ be a homomorphism to a cyclic group of prime power order which vanishes on $\ker(A\to B)$.  Then there exists an integer $l \geqslant k$ and a homomorphism $\chi'\colon B\to \Z_{q^l}$, such that the following diagram commutes:
\[ \xymatrix{ A\ar[r]\ar[d] & \Z_{q^k} \ar[d]^{\cdot q^{l-k}} \\ B\ar[r] & \Z_{q^l}.}\]
For example, this observation can be proved using the classification of finite abelian groups.

Now let
\[P := \ker(TH_1(X^\vph,Y_a^\vph;\Z) \to TH_1(E_C^{\vph},Y_a^\vph;\Z)).\]
By Proposition \ref{prop:linkform} this is a metaboliser of the linking form  $\lambda_L$.  Now let $\chi$ be as in the statement of the lemma.
We apply Lemma \ref{lem:isomsplit} to $J = L$. We write \[F:=H_1(X_L^\vph, Y_a^{\vph};\Z)/\mbox{torsion}\]
and we obtain the following commutative diagram
\[ \xymatrix{H_1(X_L^\varphi)\ar[d]\ar[r] &
H_1(X_L^\varphi,Y_a^\varphi)\ar[d]^{i_*}\ar[r]^-\cong  &
TH_1(X_L^\varphi,Y_a^\varphi)\oplus F \ar[d]^{i_*\oplus \id}\ar[r]^-{\chi_T\oplus \chi_F}&
\Z_{q^k}    \\
H_1(E_C^\varphi)\ar[r] & H_1(E_C^\varphi,Y_a^\varphi)\ar[r]^-\cong &
TH_1(E_C^\varphi,Y_a^\varphi)\oplus F.}\]
Note that since $\chi$ vanishes on $P$, we can apply the above observation to
$A:= TH_1(X_L^\varphi,Y_a^\varphi)\to  B:= TH_1(E_C^\varphi,Y_a^\varphi)$ and
to $\chi_T$ to obtain
a homomorphism $\chi_T'\colon  TH_1(E_C^\varphi,Y_a^\varphi) \to \Z_{q^l}$ with the appropriate properties.  The homomorphisms
\[H_1(E_C^\varphi)\to H_1(E_C^\varphi,Y_a^\varphi)\xrightarrow{\cong}
TH_1(E_C^\varphi,Y_a^\varphi)\oplus F\xrightarrow{\chi_T'\oplus q^{l-k}\cdot \chi_F}
\Z_{q^l}\]
complete the diagram as required.
\end{proof}

For a two-component link $L$ of linking number 1, we define the closed 3-manifold $M_L$ to be given by the union
\[M_L := X_L \cup_{\partial X_L} -X_H\]
of the exterior of $L$ and the exterior of the Hopf link glued together along their common boundary $\partial X_L = \partial X_H$, respecting the ordering of the link components and identifying each of the subsets $Y_a \subset \partial X_J$ for $J = L,H$.  Note that $H_1(M_L;\Z) \cong \Z^3$, and that, if $L=H$, we have $M_H \cong T = S^1 \times S^1 \times S^1$.

\begin{lemma}\label{Lemma:H_1X_Lvarphi_subgroup}
Let $\vph \colon H_1(M_L;\Z) \to A$ be a homomorphism to a finite abelian $p$-group $A$. Then $H_1(X_L^\vph;\Z)$ fits into a short exact sequence \[0 \to H_1(X_L^\vph;\Z) \to H_1(M_L^\vph;\Z) \to \Z \to 0,\]
so that $H_1(X_L^\vph;\Z)$ is a subgroup of $H_1(M_L^\vph;\Z).$
\end{lemma}

\begin{proof}
Consider the following exact sequence which arises from the Mayer-Vietoris sequence:
\[H_1((\partial X_L^1)^{\vph}) \oplus H_1((\partial X_L^2)^{\vph}) \to H_1(X_L^\vph) \oplus H_1(X_H^\vph) \to  H_1(M_L^\vph) \to \Z \to 0.\]
The $\Z$--term is given by \[\ker\big(H_0((\partial X_L^1)^{\vph}) \oplus H_0((\partial X_L^2)^{\vph}) \to H_0(X_L^\vph) \oplus H_0(X_H^\vph)\big).\]
We make the following claim.
\begin{claim}
\[\ba{rcl} \ker\big(H_1((\partial X_L^1)^{\vph}) \oplus H_1((\partial X_L^2)^{\vph}) \to H_1(X_H^\vph)\big) &\subseteq& \\ \ker\big(H_1((\partial X_L^1)^{\vph}) \oplus H_1((\partial X_L^2)^{\vph}) \to H_1(X_L^\vph)\big) & & \ea\]
\end{claim}

Since the maps $(\partial X_H^i)^{\vph} \to X_H^\vph$ are homotopy equivalences for $i=1,2$, the map $H_1((\partial X_L^1)^{\vph}) \oplus H_1((\partial X_L^2)^{\vph}) \to H_1(X_H^\vph)$ is surjective.  The claim then implies that the Mayer-Vietoris sequence for $M_L^\vph$ becomes the short exact sequence
\[0 \to H_1(X_L^\vph) \to H_1(M_L^\vph) \to \Z \to 0.\]
as required.

We now turn to the proof of the claim.
Let $\pi \colon Y^\vph \to Y$ be a finite covering: there are induced maps $\pi_* \colon H_*(Y^\vph) \to H_*(Y)$, and transfer maps $t_* \colon H_*(Y) \to H_*(Y^\vph)$.  To prove the claim, we consider the following commutative diagram:

\[\xymatrix{
H_1(X_L^\vph) \ar@<2pt>[d]^{\pi_*} & H_1((\partial X_L^1)^{\vph}) \oplus H_1((\partial X_L^2)^{\vph}) \ar[l]^-{h} \ar[r] \ar@<2pt>[d]^{\pi_* \oplus \pi_*} & H_1(X_H^\vph) \ar@<2pt>[d]^{\pi_*}\\
H_1(X_L) \ar@<2pt>[u]^{t_*} & H_1(\partial X_L^1) \oplus H_1(\partial X_L^2) \ar[l]^-f \ar[r]_-g \ar@<2pt>[u]^{t_* \oplus t_*} & H_1(X_H) \ar@<2pt>[u]^{t_*}.
}\]
Suppose that we have $$x \in \ker\big(H_1((\partial X_L^1)^{\vph}) \oplus H_1((\partial X_L^2)^{\vph}) \to H_1(X_H^\vph)\big).$$  Then $x \mapsto 0 \in H_1(X_H)$.  In the base spaces, we have that $\ker f = \ker g$, with both kernels generated by $[\mu_1 -\lambda_2]$ and $[\mu_2 - \lambda_1]$, where $\mu_i,\lambda_i$ is the meridian and longitude of the link component $i$.  Since the right hand square with down arrows is commutative, we see that $$f ((\pi_* \oplus \pi_*)(x)) = 0 \in H_1(X_L).$$ We write $y:= (\pi_* \oplus \pi_*)(x)$.  Since the left hand square is commutative with up arrows, we have that \[h((t_* \oplus t_*)(y)) = t_*(f(y)) = t_*(0) = 0 \in H_1(X_L^\vph).\]  We then observe that for tori such as $\partial X_L^i \cong S^1 \times S^1$, we have that $$t_* \circ \pi_* \colon H_1((\partial X_L^i)^\vph) \to H_1((\partial X_L^i)^\vph)$$ is given by multiplication by $p^\a$ for some $\a \geqslant 0$, which may depend on $x$.  Therefore
\[0 = h((t_* \oplus t_*)(y)) = h((t_* \oplus t_*)(\pi_* \oplus \pi_*)(x)) = h(p^\a x) = p^\a h(x)\]
 We have that the image $h(x)$ of $x$ in $H_1(X_L^\vph)$ is either $p$-primary torsion, or is zero.  To prove the claim we therefore wish to show that the latter case always holds.  We therefore proceed to show that there is no $p$-primary torsion in $H_1(X_L^\vph)$.

We use the fact that the inclusion map $\partial X_L^1 \to X_L$ induces a $\Z$-homology equivalence $H_*(\partial X_L^1;\Z) \toiso H_*(X_L;\Z)$ to deduce that there is also an isomorphism $H_*((\partial X_L^1)^{\vph};\Z_{(p)}) \toiso H_*(X_L^\vph;\Z_{(p)})$.  This deduction is an extension of Lemma \ref{lemma:Zhomequiv_implies_torsion} to $\Z_{(p)}$ coefficients.  We refer to \cite[Lemma~3.3]{cha_Hirzebuch_type_defects}, which follows from \cite[Lemma~4.3]{Le94}.  Levine shows that inverting integers coprime to $p$ in the coefficients of the group ring of a $p$-group $A$, to obtain $\Z_{(p)}[A]$, is sufficient to invert square matrices over $\Z[A]$ which are invertible over $\Z_p$.

This then implies, since $H_1((\partial X_L^1)^\vph;\Z) \cong \Z_{(p)}^2$, that there is no $p$-primary torsion in $H_1(X_L^\vph;\Z)$.  We make use of the universal coefficient theorem and the fact that $H_0((\partial X_L^1)^\vph;\Z)$ and $H_0(X_L^\vph;\Z)$ are torsion free.
This completes the proof of the claim.
\end{proof}

Suppose that $L$ is concordant to $H$ with concordance $C$.  The exterior $E_C$ of the concordance is a 4-manifold with boundary $X_L \cup (\partial X_H \times I) \cup X_H$. We can glue $E_C$ and $X_H \times I$ together along the $\partial X_H \times I$ components of their boundaries to define the manifold
\[W_C := E_C \cup (X_H \times I),\]
glued in such a way that $W_C$ is a 4-manifold with boundary $\partial W_C \cong M_L \cup -M_H$.

For $J = L,H$ we have the following diagram with exact rows given by Mayer-Vietoris sequences with $\Z$ coefficients and vertical maps induced by the inclusion maps:
\[\xymatrix @C-0.4cm{H_i(\partial X_J) \ar[r] \ar[d]_{\cong} & H_i(X_J) \oplus H_i(X_H) \ar[r] \ar[d]_{\cong} &  H_i(M_J) \ar[r] \ar[d] & H_{i-1}(\partial X_J) \ar[d]_{\cong}  \\H_i(\partial X_J \times I) \ar[r] & H_i(E_C) \oplus H_i(X_H \times I) \ar[r] & H_i(W_C) \ar[r] & H_{i-1}(\partial X_J \times I).
}\]
We already know that  the vertical maps, except possibly the maps $H_i(M_J) \to H_i(W_C)$,  are isomorphisms.
 But  the 5-lemma then implies that the maps  $H_i(M_J) \to H_i(W_C)$ are also isomorphisms, i.e.  the inclusions $M_L \to W_C$ and $M_H \to W_C$ induce $\Z$-homology equivalences.

Now let $\vph \colon H_1(M_L^\vph;\Z) \to A$ be such that $\vph$ restricted to $H_1(X_L;\Z) \subset H_1(M_L;\Z)$ is admissible.  We also call such homomorphisms admissible and we denote the corresponding cover of $M_L$ by $M_L^\vph$.

Using the isomorphism on first homology $H_1(M_L;\Z) \toiso H_1(W_C;\Z)$, the map $\varphi$ extends to maps $\varphi \colon H_1(W_C;\Z) \to A$ and $H_1(M_H;\Z) \to A$. We can therefore define the corresponding covering spaces, which we denote by $W_C^{\vph}$ and $M_H^\vph$.

\begin{proposition}\label{prop:extend2}
Let $L$ be an oriented 2-component link which is concordant to the Hopf link via a concordance $C$.  Let  $\varphi\colon H_1(M_L;\Z)\to A$ be an
admissible epimorphism onto a group of prime power order.
Suppose that we have a character $\chi \colon H_1(M_L^{\vph};\Z) \to \Z_{q^k}$ such that
$\chi|_{H_1(X_L^\vph)} \colon H_1(X_L^\vph;\Z) \to \Z_{q^k}$
extends to $\chi' \colon H_1(E_C^\vph;\Z) \to \Z_{q^l}.$
Then $\chi$ also extends to $\chi' \colon H_1(W_C^\vph;\Z) \to \Z_{q^l}$.  In other words, the map which is represented by the dotted arrow in the following commutative diagram exists.
\[\xymatrix{H_1(X_L^\vph) \ar[rr] \ar[ddd] \ar[dr] & & H_1(M_L^\vph) \ar[ddd] \ar[dl] \\
 & \Z_{p^k} \ar[d]^{\cdot q^{l-k}} & \\
 & \Z_{p^l} & \\
 H_1(E_C^\vph) \ar[ur] \ar[rr] & & H_1(W_C^\vph) \ar@{-->}[ul]}\]
\end{proposition}

\begin{proof}
We consider the exact Mayer-Vietoris sequence:
\[H_1((\partial X_L^1 \times I)^{\vph}) \oplus H_1((\partial X_L^2 \times I)^{\vph}) \to\]\[ H_1(E_C^\vph) \oplus H_1(X_H^\vph \times I) \to H_1(W_C^\vph) \to \Z \to 0,\]
where the $\Z$ term is given by
\[\ker\big(H_0((\partial X_L^1 \times I)^{\vph}) \oplus H_0((\partial X_L^2 \times I)^{\vph}) \to H_0(E_C^\vph) \oplus H_0(X_H^\vph \times I) \big).\]
We apply the same argument as in the proof of Lemma \ref{Lemma:H_1X_Lvarphi_subgroup}, noting also that the inclusion induces isomorphisms $$H_*(X_L;\Z) \toiso H_*(E_C;\Z),$$ and therefore isomorphisms $$H_*(X_L^\vph;\Z_{(p)}) \toiso H_*(E_C^\vph;\Z_{(p)}).$$  This means that we also have no $p$-primary torsion in $H_1(E_C^\vph;\Z)$, so the same argument with the transfer maps does indeed apply to show that the Mayer-Vietoris sequence becomes
\[0 \to H_1(E_C^\vph) \to H_1(W_C^\vph) \to \Z \to 0.\]
We have the commutative diagram:
\[\xymatrix{0 \ar[r] & H_1(X_L^\vph) \ar[r] \ar[d] & H_1(M_L^\vph) \ar[r] \ar[d] & \Z \ar[r] \ar@{=}[d] & 0 \\
0 \ar[r] & H_1(E_C^\vph) \ar[r] & H_1(W_C^\vph) \ar[r] & \Z \ar[r] & 0.}\]
Choosing a splitting map $\Z \to H_1(M_L^\vph)$ for the top row induces a splitting of the bottom row by commutativity, so that we have:
\[\xymatrix{H_1(X_L^\vph) \oplus \Z \ar[r]^-{\cong} \ar[d]_{i_* \oplus \Id} & H_1(M_L^\vph) \ar[d] \\
H_1(E_C^\vph) \oplus \Z \ar[r]^-{\cong} & H_1(W_C^\vph).}\]
We then see that a character which extends from $H_1(X_L^\vph)$ to $H_1(E_C^\vph)$ also extends from $H_1(M_L^\vph)$ over $H_1(W_C^\vph)$, which completes the proof.
\end{proof}

\section{Obstructions}\label{section:obstructions}

\subsection{Preparation}

We quote the following proposition, again from \cite{FP10}.  This is a consequence of Theorem \ref{Thm:main_thm_FP10}, which comes from the same paper, and it will be used in key places in the construction of our obstruction theory.  For our applications we will take $\mathcal{H}'$ to have rank $3$, and $A := \Z_{p^\a} \oplus \Z_{p^\b}$.

\begin{proposition}\label{prop:lifting_homology_to_covers}
Let $q$ be a prime.  Suppose that $S,Y$ are finite CW-complexes such that there is a map $i \colon S \to Y$ which induces an isomorphism \[i_* \colon H_*(S;\Z_q) \xrightarrow{\simeq} H_*(Y;\Z_q).\]
 Let $\phi' \colon \pi_1(Y) \to \mathcal{H}'$ be a homomorphism to a torsion-free abelian group $\mathcal{H}'$
 and let $\varphi\colon \pi_1(Y) \to \mathcal{H}' \to A$ be an epimorphism onto a finite abelian group which factors through $\phi'$.
Let $Y^\varphi$ be the induced  cover of $Y$,
 let $\phi \colon \pi_1(Y^\varphi) \to \mathcal{H}'$ be the restriction of $\phi'$, let $\mathcal{H} \subseteq \mathcal{H}'$ be the image of $\phi$, and let $\a' \colon \pi_1(Y^\varphi) \to \GL(d,Q)$ be a $d$-dimensional representation to a field $Q$ of characteristic zero, such that $\a'$ restricted to the kernel of $\phi$ factors through a $q$-group.  Define $S^\varphi$ to be the pull-back cover $S^\varphi := i^*(Y^\varphi)$.  Then \[ i_*^\vph \colon H_*(S^\varphi;Q(\mathcal{H})^d) \xrightarrow{\simeq} H_*(Y^\varphi;Q(\mathcal{H})^d)\]
 is an isomorphism.
\end{proposition}

We now set the stage for the definition of our invariant.
Let $L = L_1 \cup L_2 $ be a two--component oriented link with $\lk(L_1,L_2) = 1$.  Choose a homomorphism
\[\phi' \colon \pi_1(M_L) \to H_1(M_L) \toiso \mathcal{H}'\]
onto a free abelian group of rank 3.  Choose a prime $p$, a $p$-group $A = \Z_{p^\a} \oplus \Z_{p^\b}$ for integers $\a,\b$, and an admissible homomorphism $\vph \colon \pi_1(M_L) \to H_1(M_L;\Z) \to A$.
Define \[\phi \colon \pi_1(M_L^\vph) \to \mathcal{H}'\]
to be the restriction of $\phi'$ to $$\pi_1(M_L^\vph) \cong \ker(\pi_1(M_L) \to H_1(M_L) \xrightarrow{\vph} A)$$ and let $\mathcal{H} \subseteq \mathcal{H}'$ be the image of $\phi$.  As a finite index subgroup $\mathcal{H}$ is also free abelian of rank 3.  Finally, we choose a homomorphism
\[\chi \colon \pi_1(M_L^\vph) \to H_1(M_L^\vph) \to \Z_{q^k},\]
where $q$ is a prime and $k \geqslant 1$ is an integer.  In this section, we proceed as follows.  We will define an invariant of the isotopy class of $L$, $\tau(L,\chi) := \tau(L,\phi',\vph,\chi)$, as a Witt class of a certain combination of intersection forms on a 4-manifold $W$ whose boundary is a collection of copies of $M_L \sqcup -M_H$.  We will then show that our invariant is well-defined, in that it does not depend on the choices we make to initially define it.  Our main theorem is then that this invariant gives concordance obstructions for certain judicious choices of the representation $\chi$.

\subsection{Definition of $\tau(L,\chi)$}\label{section:deftau}

Note that $X_J$, for $J=L,H$, is an irreducible $3$-manifold. It is well--known that the result of gluing two irreducible 3--manifolds along incompressible tori is again
 irreducible. It follows that  $M_L$ is also irreducible.

The homomorphism $\phi \colon \pi_1(M_L^\vph) \to \HH$ induces a map $$\phi_* \colon H_i(\pi_1(M_L^\vph);\Z) \to H_i(\mathcal{H};\Z)$$ for all $i$.
Since covering maps induce isomorphisms on homotopy groups, by the Papakyriakopoulos sphere theorem and the fact that $M_L$ is irreducible, we have that $M_L^\vph$ is an Eilenberg-Maclane space $K(\pi_1(M_L^\vph),1)$. Therefore $$H_i(\pi_1(M_L^\vph);\Z) \cong H_i(M_L^\vph;\Z).$$

\begin{claim}
The map $\phi\colon \pi_1(M_L^\varphi)\to \HH$ gives rise to an isomorphism \\$\phi_*\colon H_3(M_L^\varphi;\Z)\to H_3(\HH;\Z)$.
\end{claim}

First note that the map $$\phi' \colon \pi_1(M_L) \to \HH'$$ allows us to define a map $M_L \to T$, by elementary obstruction theory, since $T$ is a $K(\HH',1)$.  Both $M_L$ and $T$ have the same cup product structure with $\Z$ coefficients, which easily implies that $$\phi'^* \colon H^3(\HH';\Z) = H^3(T;\Z) \toiso H^3(M_L;\Z) = H^3(\pi_1(M_L);\Z)$$ is an isomorphism.  By the universal coefficient theorem and Poincar\'{e} duality, $$TH_2(M_L;\Z) \cong TH^1(M_L;\Z) \cong TH_0(M_L;\Z) \cong 0.$$
Similarly,
$$TH_2(T;\Z) \cong 0.$$  This implies, again by the universal coefficient theorem, making choices of fundamental classes $[M_L]$ and $[T]$, that $$H^3(M_L;\Z) \cong H_3(M_L;\Z)$$ and $$H^3(T;\Z) \cong H_3(T;\Z).$$  The representation $\phi'$ therefore induces an isomorphism $$\phi'_* \colon H_3(M_L;\Z) \toiso H_3(T;\Z).$$  We then restrict to the finite index subgroups $\pi_1(M_L^\vph) \subseteq \pi_1(M_L)$ and $\HH \subseteq \HH'$.  The transfer maps induce isomorphisms
 \[t_* \colon H_3(M_L;\Z) \toiso H_3(M_L^\vph;\Z)\]
and
\[t_* \colon H_3(T;\Z) \toiso H_3(T^\vph;\Z),\]
where we denote by $T^\vph$ the corresponding finite cover of $T$, which is nonetheless abstractly homeomorphic to $T$.
We thus induce an isomorphism $$H_3(\pi_1(M_L^\vph);\Z) = H_3(M_L^\vph;\Z) \toiso H_3(\HH;\Z).$$
This concludes the proof of the claim.

Now let $T := S^1 \times S^1 \times S^1$ be the 3-torus. It follows from the above claim that there exists an isomorphism $\psi \colon \pi_1(T) \cong H_1(T;\Z) \toiso \HH$ such that
\begin{equation}\label{eqn:condition_on_psi}
\phi_*([M_L^\vph]) = \psi_*([T]) \in H_3(\HH) \cong \Z.
\end{equation}

We now follow the Casson-Gordon program \cite{CassonGordon} for defining obstructions.
We first have the following claim:

\begin{claim}
There exists an integer $s \geqslant 0$ and a 4-manifold $W$ with a representation $\kappa \times \Phi \colon \pi_1(W) \to \Z_{q^k} \times \HH$ and an orientation preserving homeomorphism
\[I \colon \bigsqcup_{q^s}\, (M_L^\vph \sqcup -T)  \to \partial W\]
such that, restricted to each component of the boundary,
\[(\kappa \times \Phi)| \circ I_* = \chi \times \phi \colon \pi_1(M_L^\vph) \to \Z_{q^k} \times \HH \]
and
\[(\kappa \times \Phi)| \circ I_* = \tr \times \psi \colon \pi_1(T) \to \Z_{q^k} \times \HH,\]
where $\tr$ is the trivial homomorphism.
\end{claim}

First note that
\[(M_L^\vph, \chi \times \phi) \cup -(T,\tr \times \psi) \]
defines an element in $\Omega_3(\Z_{q^k} \times \HH)$.
By the Atiyah-Hirzebruch spectral sequence,
\[\Omega_3(\Z_{q^k} \times \HH) \cong H_3(\Z_{q^k} \times \HH ;\Omega_0)  \cong H_3(\Z_{q^k} \times \HH ;\Z).\]
By the K\"{u}nneth theorem,
\[ H_3(\Z_{q^k} \times \HH)  \cong  \bigoplus_{j=0}^3 \, H_j(\Z_{q^k}) \otimes H_{3-j}(\HH) \]
 \[ \ba{rcl} &\cong&  (H_0(\Z_{q^k}) \otimes H_3(\HH)) \oplus (H_1(\Z_{q^k}) \otimes H_2(\HH)) \oplus (H_3(\Z_{q^k}) \otimes H_0(\HH)) \\
 &\cong&  (\Z \otimes H_3(\HH)) \oplus (\Z_{q^k} \otimes \Z^3) \oplus (\Z_{q^k} \otimes \Z)\ea\]
since $H_2(\Z_{q^k};\Z) \cong 0$.  Therefore
\[\Omega_3(\Z_{q^k} \times \HH) \cong H_3(\HH;\Z) \oplus \bigoplus_4\, \Z_{q^k}.\]
By (\ref{eqn:condition_on_psi}), the orientations are such that the image of
\[(M_L^\vph, \chi \times \phi) \cup -(T,\tr \times \psi)\] in $H_3(\HH;\Z)$ is zero.  Therefore $M_L \sqcup -T$, with the representations, is $q$-primary torsion in $\Omega_3(\Z_{q^k} \times \HH)$.
This concludes the proof of the claim.

We choose an inclusion $\Z_{q^k} \hookrightarrow S^1 \subseteq \C$.  Then, using the composition of ring homomorphisms
\[\Z[\pi_1(W)] \to \Z[\Z_{q^k} \times \HH] = \Z[\Z_{q^k}][\HH] \to \Z[\xi_{q^k}][\HH] \to \Q(\xi_{q^k})(\HH),\]
where $\xi_{q^k} \in S^1 \subseteq \mathbb{C}$ is a primitive $q^k$-th root of unity, we can define the twisted intersection form $\lambda_{\Q(\xi_{q^k})(\HH)}$ to be the non-singular part of the middle dimensional intersection form on $H_2(W;\Q(\xi_{q^k})(\HH))$.  If we also consider the non-singular part of the ordinary intersection form $\lambda_\Q$ on rational homology, we obtain another intersection form, which we can also consider as a form of $\Q(\xi_{q^k})(\HH)$-modules by tensoring up via the inclusion $\Q \to \Q(\xi_{q^k})(\HH)$. We define $\tau(L,\chi)$ to be
\[\big(\lambda_{\Q(\xi_{q^k})(\HH)}(W) - \Q(\xi_{q^k})(\HH) \otimes_{\Q} \lambda_\Q\big) \otimes \frac{1}{q^s} \in L^0(\Q(\xi_{q^k})(\HH)) \otimes_{\Z} \Z[1/q].\]

We first show that $\tau(L,\chi)$ is well-defined.  As usual when defining such invariants, the following lemma simplifies several key arguments.  In particular, it implies that the intersection form $\lambda_{\Q(\xi_{q^k})(\HH)}(W)$ is non-singular on the whole of $H_2(W;\Q(\xi_{q^k})(\HH))$.  When $J=L$ we use $\a' = \chi$ and when $J=H$ we use the trivial representation for $\a'$.

\begin{lemma}\label{lemma:homology_M_L_vanishes}
Let $p,q$ be primes.  For any two-component link $J$ with $\lk(J_1,J_2) = 1$, with an isomorphism $\phi' \colon H_1(M_J;\Z) \to \HH'$, a homomorphism $\vph \colon H_1(M_J;\Z) \to \HH' \to A$ to a finite abelian $p$-group, and a representation $$\a' \colon H_1(M_J^\vph;\Z) \to \Z_{q^k} \hookrightarrow S^1 \subset \C,$$ we have that
\[H_*(M_J^\vph;\Q(\xi_{q^k})(\HH)) \cong 0,\]
where $\xi_{q^k}$ is a primitive $q^k$th root of unity.
\end{lemma}

\begin{proof}
Since the linking number is 1, the inclusion maps $\partial X_J^1 \to X_J$ are $\Z$-homology equivalences.  We wish to apply Proposition \ref{prop:lifting_homology_to_covers} with $S = S^1 \times S^1$ to see that
\[H_*(X_J^\vph;\Q(\xi_{q^k})(\HH)) \cong H_*((\partial X_J^1)^\vph;\Q(\xi_{q^k})(\HH)) \cong 0.\]
In order to apply Proposition \ref{prop:lifting_homology_to_covers}, we need to note that $\Z[\xi_{q^k}]$ is a domain, so the cyclotomic field $Q:= \Q(\xi_{q^k})$ is a field of characteristic zero, and that $\a'$ factors through a $q$-group.

The Mayer-Vietoris sequence of $M_J^\vph = X_J^\vph \cup_{\partial X_J^\vph} X_H^\vph$ with twisted coefficients then yields the desired result.
\end{proof}

\begin{lemma}\label{lemma:independence_of_tau_on_choices}
The Witt group class $\tau(L,\chi) \in L^0(\Q(\xi_{q^k})(\HH)) \otimes_{\Z} \Z[1/q]$ depends neither on the choice of 4-manifold $W$ nor on the choice of isomorphism $\psi \colon \pi_1(T) \to \HH$.
\end{lemma}

\begin{proof}
To see that the object $\tau(L,\chi)$ as defined is independent of the choice of 4-manifold $W$, and so is an invariant, we need, by Novikov additivity for the rational intersection form and by Lemma \ref{lemma:homology_M_L_vanishes} for twisted coefficients, to see that for a closed manifold $$V := W \cup_{q^sM_L^\vph \sqcup q^s T} -W',$$ the $L$-group classes
\[\lambda_{\Q(\xi_{q^k})(\HH)}(V) \mbox{ and } \lambda_\Q(V) \in L^0(\Q(\xi_{q^k})(\HH)) \otimes_{\Z} \Z[1/q]\]
coincide.  This follows again from the Atiyah-Hirzebruch spectral sequence.
 More precisely, this follows from the argument of \cite[Lemma~2.1]{cha_Hirzebuch_type_defects},
 the only difference in our case is that  $H_4(\Z_{q^k} \times \HH)$ is not zero, but is a finite $q$-group, and Cha's argument thus implies that
\[\Omega_4(\Z_{q^k} \times \HH) \otimes_{\Z} \Z[1/q] \cong \Omega_4 \otimes_{\Z} \Z[1/q].\]
Therefore, modulo $q$-primary torsion, the twisted intersection form of $V$ is Witt equivalent to the ordinary intersection form of $V$.

The obstruction does not depend on the choice of isomorphism \\$\psi \colon \pi_1(T) \to \HH$: simply choose a different orientation preserving self-homeomorphism of $T$ corresponding to the difference in two choices $\psi$ and $\psi'$, and glue a product cobordism onto $W$ for each copy of $T$ in $\partial W$, using the new homeomorphism at the new boundary.  We are using here the fact that any element of $\SL(3,\Z)$ can be realised by a homeomorphism of $T$, since $T \cong \R^3/\Z^3$.
\end{proof}

\subsection{Main obstruction theorem}

In order to prove Theorem \ref{thm:main_theorem} below, we will need the following lemma to justify having taken the trivial character on $\pi_1(T)$ to define $\tau(L,\chi)$.

\begin{lemma}\label{Lemma:any_character_cobordant_to_trivial_one}
There is a cobordism of $q^l$ copies of the 3-torus $T$ with any character $$\chi \colon \pi_1(T) \toiso H_1(T;\Z) \to \Z_{q^l},$$ via a cobordism $W_{\chi}$ which is depicted in Figure \ref{Fig:link_conc_hopf_link2}, to $q^l$ copies of $T$ with the trivial character, with an isomorphism $\psi \colon \pi_1(T) \to \mathcal{H}$ extending over $W_\chi$.  Moreover $W_{\chi}$ has trivial intersection form both with $\Q(\xi_{q^l})(\HH)$-coefficients twisted using the extensions of $\psi$ and $\chi$, $\chi \times \psi \colon \pi_1(W_\chi) \to \Z_{q^l} \times \HH$, and with rational coefficients.
\end{lemma}

\begin{proof}
First note that $S^1$ with a trivial character $(S^1,\tr)$ is cobordant on the one hand to $q^l$ copies of $S^1$, again with a trivial character, via a cobordism $V_{\tr}$.  On the other hand, $(S^1,\tr)$ is also cobordant to $q^l$ copies of $S^1$ with the same non-trivial character $\chi \colon \pi_1(S^1) \to \Z_{q^l}$ on each copy, via a cobordism $V_{\chi}$.  If we glue together $V_{\tr}$ and $V_{\chi}$, and cross the resulting cobordism with $S^1 \times S^1$, we obtain a cobordism from $q^l$ copies of $T$ to itself, with any character changed to the trivial character on the first $S^1$ of $S^1 \times S^1 \times S^1 = T$.  That is, we define
\[W_\chi := (V_{\chi} \cup_{S^1} V_{\tr}) \times S^1 \times S^1.\]
$W_\chi$ is just a union of several pairs of trousers (pants), crossed with a torus.

Part of the definition of a cobordism $W_\chi$ is a collection of homeomorphisms $T \toiso (\partial W_\chi)_i$ for each component $(\partial W_\chi)_i, i= 1 ,\dots, q^l$, of $\partial W_\chi$ with the non-trivial character associated to it.  We pick a homeomorphism, which we use for each $i = 1 ,\dots ,q^l$, $T \toiso (\partial W_\chi)_i$ such that the induced character $\chi \colon \pi_1((\partial W_\chi)_i)  \to \Z_{q^l}$ factors through \[\pi_1((\partial W_\chi)_i) \toiso \pi_1(S^1) \times \pi_1(S^1 \times S^1) \xrightarrow{\pr} \pi_1(S^1) \cong \Z \to \Z_{q^l}.\]   We therefore only need to change the character on one $S^1$ component of $T$ in order to arrive at the trivial character associated to $(\partial W_\chi)_i$ for $i=q^l +1, \dots,2q^l$.  This is precisely what $W_\chi$ achieves.

Since $W_\chi$ is a union of pairs of trousers, crossed with a torus, and since we have the same isomorphism $\psi \colon \pi_1(T) \toiso \HH$ for each copy of $T$ at the end of the cobordism with non-trivial character, then $\psi$ extends over $W_\chi$.  If we require a different isomorphism $\psi$ at the end of $W_\chi$ with trivial character, then, as in Lemma \ref{lemma:independence_of_tau_on_choices}, we can glue on a product cobordism to each of the copies $(T,\tr \times \psi) \subset \partial W_\chi,$ of $T$ with the trivial character, using a different self-homeomorphism of $T$.

The map $H_2(\partial W_\chi) \to H_2(W_\chi)$ is surjective.  The non-singular part of the intersection form on $H_2(W_\chi)$ is therefore trivial.
On the other hand, we have by the Eckmann--Shapiro lemma, that
\[ \ba{rcl} H_2(W_\chi;\Q(\xi_{q^l})[\HH])&=&H_2( (V_{\chi} \cup_{S^1} V_{\tr}) \times S^1 \times S^1;\Q(\xi_{q^l})[\HH])\\
&=&H_2( (V_{\chi} \cup_{S^1} V_{\tr});\Q(\xi_{q^l})[t^{\pm 1}]),\ea \]
but the last group is zero since $V_{\chi} \cup_{S^1} V_{\tr}$ is homotopy equivalent to a 1--complex.
It thus follows that the twisted intersection form on $W_\chi$ is also trivial.
\end{proof}

The following is our main obstruction theorem.  Note that we use the inclusion $\Q(\xi_{q^k}) \hookrightarrow \C$ in order to obviate dealing with the complications which arise from requiring higher prime powers to extend a character over a 4-manifold.

\begin{theorem}\label{thm:main_theorem}
Suppose that $L$ is concordant to the Hopf link.  Then for any admissible homomorphism $\varphi \colon H_1(M_L) \to A$ to a finite abelian $p$-group $A$, for a prime $p$, there exists a metaboliser $P = P^\bot$ for the linking form
\[TH_1(X_L^\varphi,Y_a^\varphi)\times TH_1(X_L^\varphi,Y_{a}^\varphi)\to \Q/\Z\]
 with the following property: for any character of prime power order $\chi \colon H_1(M_L^\varphi) \to \Z_{q^k}$ which satisfies that $\chi|_{H_1(X_L^{\vph})}$ factors through \[\chi|_{{H_1(X_L^{\vph})}} \colon H_1(X_L^\vph) \to H_1(X_L^\vph,Y_a^\vph) \xrightarrow{\delta} \Z_{q^k}\] and that $\delta$ vanishes on $P$, we have that
\[\tau(L,\chi) = 0 \in L^0(\C(\HH)) \otimes_{\Z} \Z[1/q],\]
using the inclusion $\Q(\xi_{q^k}) \hookrightarrow \C$ to define the map $L^0(\Q(\xi_{q^k})(\HH)) \to \\L^0(\C(\HH))$.
\end{theorem}

\begin{proof}
As above let $E_C$ be the exterior of the concordance.  The conclusion regarding the $\Q/\Z$-valued linking form was Proposition \ref{prop:linkform}: a metaboliser is given by
\[P := \ker(TH_1(X^\vph,Y_a^\vph;\Z) \to TH_1(E_C^{\vph},Y_a^\vph;\Z)).\]  Let $\chi$ satisfy the conditions set out.  As in Section \ref{section:extend}, let $W_C := E_C \cup X_H \times I$.  By Propositions \ref{prop:extend} and \ref{prop:extend2}, we see that our character extends from $\pi_1(M_L^\vph)$ to $\pi_1(W_C^{\varphi})$, if we allow it to take values in $\Z_{q^l}$ for a potentially higher prime power: $l \geqslant k$.  We include $\Q(\xi_{q^k}) \subset \C$ so that our characters always extend.

Using the inverse of the inclusion induced isomorphism $H_1(M_L;\Z) \toiso H_1(W_C;\Z)$, the map $\phi' \colon \pi_1(M_L) \to H_1(M_L) \to \HH'$ extends to \[\phi' \colon \pi_1(W_C) \to H_1(W_C) \toiso H_1(M_L) \xrightarrow{\phi'} \HH'.\]
As $\vph \colon \pi_1(M_L) \to H_1(M_L) \to A$ extends to $\vph \colon \pi_1(W_C) \to A$ using the same isomorphism $H_1(W_C) \toiso H_1(M_L)$, the restriction of $\phi'$ to \[\ker(\vph \colon \pi_1(W_C) \to H_1(W_C) \toiso H_1(M_L) \to A) \cong \pi_1(W_C^\vph)\] yields a homomorphism $\phi \colon \pi_1(W_C^\vph) \to \HH$, with $\HH \subset \HH'$ the image of $\phi \colon \pi_1(M_L^\vph) \to \HH'$ as above: since both $\phi \colon \pi_1(M_L^\vph) \to \HH$ and $\phi \colon \pi_1(W_C^\vph) \to \HH$ factor through $\ker(\vph \colon H_1(M_L) \to A)$, we indeed have that $$\im(\phi \colon \pi_1(W_C^\vph) \to \HH')\subseteq \im(\phi \colon \pi_1(M_L^\vph) \to \HH') = \HH.$$

Since both $\chi$ and $\phi$ extend to $W_C^\vph$, we can now  use $(W_C^\vph,\chi \times \phi)$ to calculate $\tau(L,\chi)$.

As $\chi$ extends to $\pi_1(W_C^\vph)$, it also restricts to a map $\chi \colon \pi_1(T) \to \Z_{q^l}$, since $T \cong M_H^\vph$.  We need to know that for any character on $\pi_1(T) \cong \Z^3$, there is a cobordism from $q^l$ copies of $(T,\chi)$ to $q^l$ copies of $(T,\tr)$, since we used the trivial character to define $\tau(L,\chi)$.  This cobordism $W_{\chi}$ was described in Lemma \ref{Lemma:any_character_cobordant_to_trivial_one}.  We take $q^l$ copies of $W_C$, and glue them to $W_\chi$: this gives us a cobordism from $q^l \cdot M_L^\vph$ to $q^l \cdot T$, which we can then use to calculate $\tau(L,\chi)$: see Figure \ref{Fig:link_conc_hopf_link2}.  We saw in Lemma \ref{Lemma:any_character_cobordant_to_trivial_one} that the intersection forms of $W_\chi$ are trivial.

\begin{figure}[h]
    \begin{center}
 {\psfrag{A}{$(M_L^\vph,\chi)$}
 \psfrag{B}{$(T,\chi)$}
 \psfrag{C}{$(T,\tr)$}
 \psfrag{D}{$W_\chi$}
 \psfrag{W}{$(W_C^\vph,\chi)$}
 \includegraphics[width=12cm]{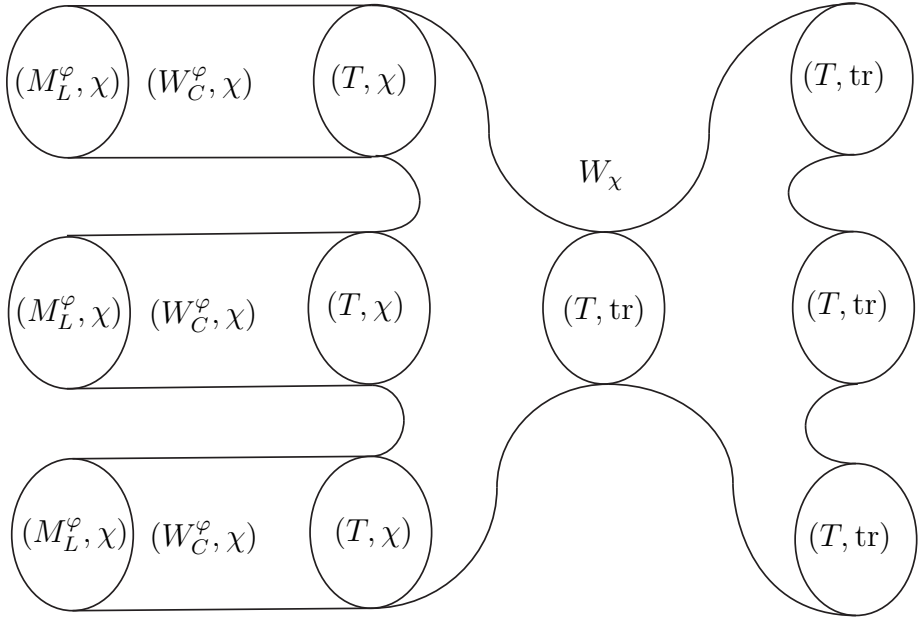}
 }
 \caption{The cobordism with which we calculate $\tau(L,\chi)$ ($q^l = 3$).}
 \label{Fig:link_conc_hopf_link2}
 \end{center}
\end{figure}

We therefore need to see that the twisted and ordinary Witt classes in $L^0(\C(\HH)) \otimes_{\Z} \Z[1/q]$ which arise from $W_C^\vph$ vanish.

We again apply Proposition \ref{prop:lifting_homology_to_covers}, with $S = M_L$ and $Y = W_C$, to see, since $M_L \to W_C$ is a $\Z$-homology equivalence, that
\[H_*(W_C^\vph;\Q(\xi_{q^l})(\HH)) \cong H_*(M_L^\vph;\Q(\xi_{q^l})(\HH)) \cong 0.\]
Here we also use the conclusion of Lemma \ref{lemma:homology_M_L_vanishes}, so we are really applying Proposition \ref{prop:lifting_homology_to_covers} twice.  Trivial homology groups necessarily carry Witt trivial forms; the rational intersection form is also trivial since \[H_2(W_C^\vph;\Q) \subseteq \im(H_2(M_L^\vph;\Q) \to H_2(W_C^\vph;\Q))\]
by Lemma \ref{lemma:Zhomequiv_implies_torsion} applied with $X=M_L$ and $E=W_C$.
Therefore, including $\Q(\xi_{q^l}) \subset \C$, we have \[\tau(L,\chi) = 0 \in L^0(\C(\HH)) \otimes_{\Z} \Z[1/q]\] as claimed, since $\C$ is flat over $\Q(\xi_{q^l})$.
\end{proof}

\section{Examples}\label{section:examples}

\subsection{Signature invariants}

Let $\HH$ be a free-abelian group. In this section we will  show how to recognise non--trivial elements in $L^0(\C(\HH))$.
 Let $P$ be a Hermitian matrix over $\C[\HH]$.
Let $\eta\colon \HH\to S^1$ be any character. Note that $\eta$ gives rise to a ring homomorphism $\C[\HH]\to \C$. We denote by $\eta(P)$ the Hermitian matrix over $\C$ obtained by applying the above ring homomorphism to all entries of $P$. We can then consider the signature $\operatorname{sign}(\eta(P))$ of the matrix $\eta(P)$.

Let $\tau \in L^0(\C(\mathcal{\HH}))$. We represent $\tau$ by a Hermitian matrix $P$ over $\C[\HH]$ and we define
\[\sigma(P):=\int_{\eta\in \operatorname{Hom}(\HH,S^1)} \operatorname{sign}(\eta(P)),\]
where we equip the torus $\operatorname{Hom}(\HH,S^1)$ with the Haar measure.

\begin{lemma}\label{lem:sigmatau}
The real number $\sigma(P)$ is independent of the choice of $P$ representing a given $\tau \in L^0(\C(\HH))$.
\end{lemma}

Before proving Lemma \ref{lem:sigmatau}, note that it allows us to define an invariant $\sigma(\tau) := \sigma(P)$ for some choice of representative $P$ of $\tau$.  Finally given $\tau \otimes q \in L^0(\C(\mathcal{\HH}))\otimes \Q$ we define
\[ \sigma(\tau\otimes q)=q\cdot \sigma(\tau).\]

\begin{proof}[Proof of Lemma \ref{lem:sigmatau}]
Since $\sigma$ is additive it suffices to show that $\sigma:L^0(\C(\mathcal{\HH}))\to \Z$ is well--defined.
Let $\tau \in L^0(\C(\mathcal{\HH}))$.
Let $P$ and $P'$ be two Hermitian matrices over $\C[\HH]$ which represent $\tau$. We let
\[ B = \begin{pmatrix} \, 0 \, & \, 1\, \\ \,1 \, & \, 0 \,\end{pmatrix},\]
and given $n\in \N$ we denote by $nB$ the diagonal sum of $n$ copies of $B$.
It is well-known that there exist $n$ and $n'$ such that $P\oplus nB$ and $P'\oplus n'B$ are congruent over $\C[\HH]$, i.e. there exist matrices $Q,Q'$ over $\C[H]$ with non-zero determinant, such that
\[ Q(P\oplus nB)Q^t=Q'(P'\oplus n'B)(Q')^t.\]
A priori the conjugation matrices $Q$ are over $\C(\HH)$, but only one is required.  We need two matrices $Q$ and $Q'$ here in order to ensure that both are over $\C[\HH]$ rather than $\C(\HH)$.
A straightforward calculation shows that
\[  \ba{rcl} \operatorname{sign}(\eta(P))&=&\operatorname{sign}(\eta(Q(P\oplus nB)Q^t))\\
&=&\operatorname{sign}(\eta(Q'(P'\oplus n'B)(Q')^t))\\
&=&\operatorname{sign}(\eta(P')),\ea \]
for all $\eta$ such that  $\det(\eta(Q))\ne 0$ and such that $\det(\eta(Q'))\ne 0$.
Put differently, we have $ \operatorname{sign}(\eta(P))= \operatorname{sign}(\eta(P'))$ for all $\eta \in \operatorname{Hom}(\HH,S^1)$ outside a set of measure zero.
\end{proof}

\subsection{The satellite construction}

Let $L$ be a link, let $K$ be a knot and let $\g \subset X_L$ be a curve which is unknotted in $S^3$.  Define $X_K:= S^3 \setminus \nu K$.  We form the space
\[S^3 \setminus \nu \g \cup_{\theta} X_K \cong S^3,\]
where $\theta \colon \partial(S^3 \setminus \nu\g) \to \partial X_K$ is a homeomorphism which sends the meridian of $\gamma$ to the zero-framed longitude of $K$ and the zero-framed longitude of $\g$ to a meridian of $K$.  The image of $L$ is the \emph{satellite link} $S(L,K,\g)$.
When $L,K$ and $\g$ are understood, then we just write  $S= S(L,K,\g)$.

Note that the canonical degree one map $S^3\sm \nu K\to S^1\times D^2$ gives rise to a canonical degree one map $S^3\sm \nu S\to S^3\sm \nu L$.
Let $\vph\colon H_1(X_L;\Z)\to A$ be an epimorphism onto a finite group. We denote the induced homomorphism $H_1(X_S;\Z)\xrightarrow{\cong} H_1(X_L;\Z)\to A$ by $\vph$ as well.
  We now have the following lemma, which says that in many cases the satellite construction has no effect on the homology of finite covers or on the linking form.

\begin{lemma}\label{lem:sameh1finitecovers}
Let $L,K$ and $\g$ be as above such that in addition $\g$ is null homologous. Let $\vph\colon H_1(X_L;\Z)\to A$ be an epimorphism onto a finite group.   Then the canonical degree one map $X_S\to X_L$ induces an isomorphism
\[TH_1(X_S^\vph,Y_a^\vph) \toiso TH_1(X_L^\vph,Y_a^\vph)\]
and the linking form of Section \ref{section:defn_of_linking_forms} is preserved.
\end{lemma}

\begin{proof}
We write $k=|A|$.
Since $\g$ is null homologous, it lifts to curves $\wt{\g}_1,\dots\wt{\g}_{k} \subset X_L^\vph$.  Then
\[X_S^\vph \cong \cl(X_L^\vph \setminus \cup_{i=1}^{k} \nu \wt{\g}_i) \cup_{\cup_{i=1}^{k} \partial X_K} \bigcup\limits_{i=1}^{k} X_K.\]
We now refer to the argument of \cite[Lemma~10.3]{Friedl}, the ideas of which originate from \cite{Litherland84}.  This argument carries over to our situation with translation of notation, and noting that $H_1(X_L^\vph) \toiso H_S(X_L^\vph)$ implies that \[H_1(X_L^\vph,Y_a^\vph) \toiso H_1(X_S^\vph,Y_a^\vph)\]
by the five lemma, and the fact that the boundary of $X_S$ is the image of the boundary of $X_L$ in the satellite construction.
\end{proof}

The key technical advantage of the satellite construction is that it allows us to calculate
$\sigma(\tau(S,\chi))$ in terms of $\tau(L,\chi)$ and the Levine--Tristram signatures of $K$.
Recall that given any knot $K$ and $\omega\in S^1$ the Levine--Tristram signature $\sigma(K,\omega)\in \Z$ is defined
and that it can be calculated using any Seifert matrix of $K$.
It is well--known that $\sigma(K,\omega)=0$ for any slice knot $K$ and any prime power root of unity.
We refer to \cite{Levine} and \cite{Tristram69} for details.

The following theorem can easily be proved using the ideas and methods of Litherland \cite{Litherland84};  see also \cite[Theorem~10.5]{Friedl}.
Similar results have been obtained by many authors over the years, and we will therefore not provide a proof.

\begin{theorem}\label{thm:litherland}
Let $L$ be a 2-component link with linking number one and let $\g$ be a null-homologous curve in $S^3\sm \nu L$ which is unknotted in $S^3$.
Let $K$ be a knot. We write $S:=S(L,K,\g)$.  Let $\vph:H_1(X_L;\Z)\to A$ be an epimorphism onto a finite group of order $k$ and let
\[\chi \colon TH_1(X_S^\vph,Y_a^\vph) \toiso TH_1(X_{L}^\vph,Y_a^\vph) \to S^1\]
be a character and let $\phi \colon H_1(M_S^\vph) \toiso H_1(M_L^\vph) \to \mathcal{H}$ be an epimorphism onto a free abelian group of rank 3.  Suppose that $\g$ lifts to $\wt{\g}_1, \dots \wt{\g}_{k}$.  Then
\[\sigma(\tau(S,\chi)) = \sigma(\tau(L,\chi)) + \sum_{i=1}^{k}\, \sigma(K,\omega_i),\]
where $\omega_i = \chi([\wt{\g}_i])$ and $[\wt{\g}_i]$ is the image of $\wt{\g}_i$ in $TH_1(X_L^\vph,Y_a^\vph)$.
\end{theorem}

%



\subsection{Example}

We consider the link $\mathcal{L}$
depicted in Figure \ref{Fig:example1}. Note that $\mathcal{L}$  is concordant to the Hopf link via a ribbon move.  It has non-trivial multi-variable Alexander polynomial
\[ \Delta_{\mathcal{L}}(s,t) = (ts+1-s)(ts+1-t),\]
which we computed using a Wirtinger presentation of the link group and the Fox calculus.

\begin{figure}[h]
 \begin{center}
 \includegraphics [width=6cm] {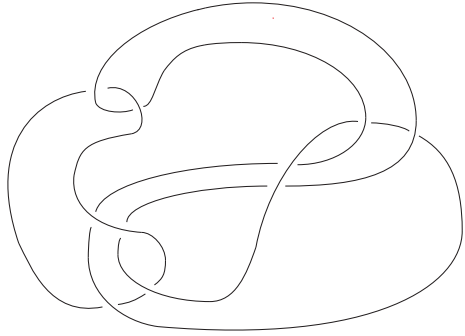}
 \caption {A link $\mathcal{L}$ which is concordant to the Hopf link, with non-trivial Alexander polynomial.}
 \label{Fig:example1}
 \end{center}
\end{figure}

We take $\vph$ to be the admissible homomorphism given by two projections $\varphi \colon H_1(X_\mathcal{L};\Z) \cong \Z \oplus \Z \to \Z_2 \oplus \Z_2$.
Using the theorem in the appendix we can calculate that $TH_1(X_{\mathcal{L}}^\vph,Y_a^\vph)$ of order 9.
In fact we have   a Maple program which calculates the homology explicitly. It starts with the Wirtinger presentation of $\pi_1(X_\mathcal{L})$, and uses it to construct the boundary maps in the chain complex
$C_*(X_\mathcal{L}^\vph;\Z)$, by using the regular representation at the end of the composition:
\[\pi_1(X_\mathcal{L}) \to H_1(X_\mathcal{L}) \cong \Z \oplus \Z \to \Z_2 \oplus \Z_2  \to  \aut(\Z[\Z_2\oplus \Z_2])\to \GL(4,\Z),\]
which sends
\[(1,0)  \mapsto  \left(\ba{cccc} \, 0 \, &  \,1\, &\, 0\, & \,0\, \\\, 1\, &\, 0\, & \, 0\, &\, 0\, \\\, 0\, &\, 0\, &\, 0\, &\, 1\, \\\, 0\, &\, 0\, &\, 1\, &\, 0\, \ea \right)
\mbox{ and }(0,1)  \mapsto  \left(\ba{cccc} \, 0\, &\, 0\, &\, 1\, &\, 0\, \\\, 0\, &\, 0\, &\, 0\, &\, 1\, \\\, 1\, &\, 0\, &\, 0\, &\, 0\, \\\, 0\, &\, 1\, &\, 0\, &\, 0\, \ea \right).\]
By doing row and column operations simultaneously on matrices representing the boundary maps $\partial_2 \colon C_2 \to C_1$ and $\partial_1 \colon C_1 \to C_0$, we directly compute that
\[H_1(X_\mathcal{L}^\vph;\Z) \cong \Z_9 \oplus \Z \oplus \Z,\]
which by Corollary \ref{cor:h1yasplits} implies that:
\[TH_1(X_\mathcal{L}^\vph,Y_a^\vph;\Z) \cong \Z_9.\]
Using  our Maple program  we augment the matrix representing $\partial_1$ with an identity matrix, whose task is to keep track of the basis changes.  With this extra information we can relate a generator of the $\Z_9$ summand of $TH_1(X_\mathcal{L}^\vph,Y_a^\vph)$ to our original Wirtinger generators, which we understand geometrically and can exhibit in the diagram of $\mathcal{L}$.
Using this approach we obtain  the curve $\g$ in Figure \ref{Fig:link_conc_hopf_link3}
which has the following three properties:
\bn
\item $\g$ represents a generator of $TH_1(X_\mathcal{L}^\vph,Y_a^\vph)$;
\item there exists a Seifert surface $F$ for $\mathcal{L}$, such that $\g$ and $F$ do not intersect (given by Seifert's algorithm and tubing along $\g$); and
\item the curve $\g$ forms the unlink with either of the two components of $L$.
\en
\begin{figure}[h]
    \begin{center}
 {\psfrag{A}{$\gamma$}
 \includegraphics[width=9cm]{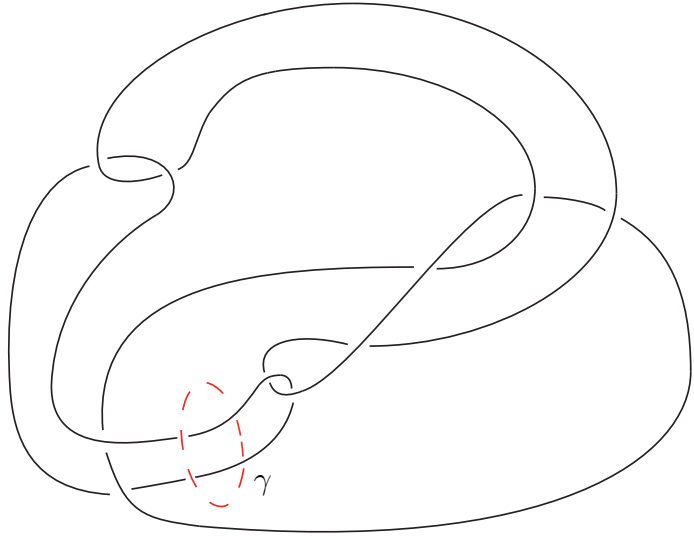}
 }
 \caption{The link $\mathcal{L}$ with the curve $\g$ shown.}
 \label{Fig:link_conc_hopf_link3}
 \end{center}
\end{figure}

\begin{claim}
Let  $K=3_1$ be the trefoil knot.  Then $S=S(\mathcal{L},K,\gamma)$ is not concordant to the Hopf link.
\end{claim}

\begin{proof}
Recall that the degree one map $X_S\to X_L$ induces an isomorphism $H_1(X_S;\Z)\to H_1(X_L;\Z)$. We denote
by $\varphi$ the induced homomorphism $H_1(X_S;\Z)\to \Z_2\oplus \Z_2$. Since $\gamma\subset X_L$ is null--homologous it follows from the previous section
that the induced map
\[ H_1(X_S^\varphi,Y_a^\varphi;\Z)\to H_1(X_L^\varphi,Y_a^\varphi;\Z)\] is an isomorphism
and in fact this map induces an isometry of linking forms.
Note that any form $\Z_9\times \Z_9\to \Q/\Z$ has a unique metaboliser.

Now let  $\chi:TH_1(X_S^\varphi,Y_a^\varphi;\Z)\to \Z_3\to S^1$ be a non--trivial character which vanishes on the unique metaboliser of the linking pairing on $TH_1(X_S^\varphi,Y_a^\varphi;\Z)\cong \Z_9$
and let $\phi \colon H_1(M_S^\vph) \to \mathcal{H}$ an epimorphism onto a free abelian group of rank 3.
By  Theorem \ref{thm:main_theorem} and Lemma \ref{lem:sigmatau}
we are done once we show that $\sigma(\tau(S,\chi))\ne 0$.

First note that $\chi$ and $\phi$ are induced from a character on \\$TH_1(X_L^\varphi,Y_a^\varphi;\Z)$ and an epimorphism on $H_1(M_L^\varphi)$,
which we denote with the same symbols.
We denote by  $\wt{\g}_1, \dots, \wt{\g}_{4}\subset X_L^\varphi$ the preimages of $\g$.
It follows from Theorem \ref{thm:litherland}
that
\[\sigma(\tau(S,\chi)) = \sigma(\tau(L,\chi)) + \sum_{i=1}^{4}\, \sigma(K,\omega_i),\]
where $\omega_i = \chi([\wt{\g}_i])$ and $[\wt{\g}_i]$ is the image of $\wt{\g}_i$ in $TH_1(X_L^\vph,Y_a^\vph)$.
Note that $ \sigma(\tau(L,\chi))=0$ since $L$ is concordant to the Hopf link and since $\chi:TH_1(X_L^\varphi,Y_a^\varphi;\Z)\to S^1$
also vanishes on the unique metaboliser of $TH_1(X_L^\varphi,Y_a^\varphi;\Z)\cong \Z_9$.
Furthermore note that any $\omega_i$ is a third root of unity, and moreover at least one $\omega_i$ is a non--trivial third root of unity since  $\g$ represents a generator of $TH_1(X_\mathcal{L}^\vph,Y_a^\vph)$.
It is well--known that $\sigma(K,1)=0$ and $\sigma(K,e^{\pm 2\pi i/3})=2$.
We thus see that $\sigma(\tau(S,\chi))\ne 0$.
\end{proof}

\subsection{Comparison with previous invariants}

We now slightly modify the example of the previous section.
By the proof of \cite[Proposition~7.4]{Fr04} there exists a knot $K$ such that $\int_{z\in S^1}\sigma(K,z)=0$ but such that
$\sigma(K,e^{2\pi i/3})=12$ and $\sigma(K,e^{4\pi i/3})=12$. We denote by $S$ the satellite link which is constructed as in the previous section, with the sole difference that we replace the trefoil knot by the knot $K$.
The argument of the previous section shows that our obstruction detects that $S$ is not concordant to the Hopf link.

We will now quickly discuss the earlier obstructions to link concordance which are known to the authors.  The enumeration here does not correspond exactly with the enumeration of link concordance obstructions which appears in the introduction.
\bn
\item The components of $S$ and $H$ are trivial knots, in particular components of $S$ do not give rise to non-trivial concordance obstructions.
\item The linking number of the two components of $S$ is $1$, which is the same as the linking number of the Hopf link. Note that for non--trivial linking numbers the higher Milnor $\ol{\mu}$-invariants are not defined; in particular they do not give further concordance obstructions.
\item A variation the proof of Lemma \ref{lem:sameh1finitecovers} shows that $\Delta_S=\Delta_\mathcal{L}\in \zst$.  It follows that
\[ \ba{rcl} \hspace{1cm}\Delta_S&=&(ts+1-s)(ts+1-t)\\
&=&s^{-1}t^{-1}(ts+1-s)(t^{-1}s^{-1}+1-s^{-1})\in \zst \ea \]
factors as a norm over $\zst$. In particular the
link concordance obstructions of Murasugi \cite{Mu67}, Kawauchi \cite{Ka77} \cite[Theorem~B]{Ka78} and Nakagawa \cite{Na78} are satisfied.
\item Since the curve $\g$ does not intersect a Seifert surface $F$ of $\mathcal{L}$ we see that the image of $F$ in
\[S^3 \setminus \nu \g \cup_{\theta} X_K \cong S^3\]
is a Seifert surface for $S$. One can easily see that the corresponding Seifert matrices for $\mathcal{L}$ and $S$ agree. Since $\mathcal{L}$ is concordant to the Hopf link it is now
straightforward to see that the link concordance obstructions of Tristram \cite{Tristram69} are satisfied.
\item Let $\mathbb{T}$ be one of the  two boundary components of $X_S$. It follows from Stalling's theorem (see \cite{St65}) applied to the homology equivalence $\mathbb{T} \to \pi_1(X_S)$, that any epimorphism from $\pi_1(X_S)$ onto a nilpotent group
(in particular a $p$-group) factors through the abelianisation of $\pi_1(X_S)$. It follows from standard arguments that all the twisted Alexander polynomials of $S$ corresponding to representations factoring through groups of prime power order are in fact determined by $\Delta_S$. Since $\Delta_S$ factors as a norm over $\zst$ it is now straightforward to verify that the link concordance obstruction of \cite{cha_twisted_2010} vanishes.
\item The fact that $\pi_1(X_S)$ has no non-abelian nilpotent  quotients also implies that the signature invariants by  Cha and Ko \cite{CK99,CK06},  Friedl \cite{Fri05}, Levine \cite{Le94}
and Smolinsky \cite{Sm89} are determined by signature invariants corresponding to abelian representations. The results of Litherland \cite{Litherland84} can easily be modified to show that such signature invariants of $S$ agree with the corresponding signature invariants of $\mathcal{L}$, which do not give any obstructions to being concordant to $H$, since $\mathcal{L}$ is concordant to $H$.
\item Similarly to the previous invariants one can also show that the obstructions of \cite[Theorem~1.4]{cha_Hirzebuch_type_defects} vanish in our case.
\item Let $\mathcal{L}_1\cup \mathcal{L}_2\subset S^3$ be a 2-component link with linking number one. We can consider the $k$-fold branched cover $Y:=Y(\mathcal{L}_1,k)$ of $S^3$ branched along $\mathcal{L}_1$. The preimage of $\mathcal{L}_2$ under the projection map $Y\to S^3$ consists of one component which we denote by $J$ (here we use that the linking number is one). If $k$ is a prime power, then $Y$ is a rational homology sphere. If $\mathcal{L}$ is concordant to the Hopf link and if $k$ is a prime power, then it is straightforward to see that $J$ is rationally concordant to the unknot.  One can therefore apply the obstructions of a knot being rationally concordant to the unknot.
    Such obstructions are given by  Levine-Tristram signatures (see Cha and Ko \cite{CK02}) and by Cochran-Orr-Teichner type Von Neumann $\rho$-invariants (see \cite{Ch07}, \cite{COT}).
    In our case it is straightforward to see that the preimage of one component of the satellite link is the satellite of a slice knot such that the infection curve is null-homologous.
    It follows from standard arguments (see above) that the Levine-Tristram signatures are zero. Finally, since also $\int_{z\in S^1}\sigma(K,z) = 0$, it follows that the obstructions of \cite{Ch07} cannot be used to show that the knot is not rationally slice.
\item As mentioned in the introduction, blowing down along one of the link components yields a knot in $S^3$ to which we can apply knot concordance obstructions, not just rational concordance obstructions (see \cite[Corollary~2.4]{CKRS10}).  It seems likely that Casson-Gordon obstructions of this knot will also obstruct $S$ from being concordant to the Hopf link.  In our example, blowing down $\mathcal{L}$ along the left hand component in the diagram (Figure \ref{Fig:link_conc_hopf_link3}) yields the slice knot $6_1$.  The 2-fold branched cover of $S^3$ along $6_1$, which we denote by $L_2^{6_1}$, has homology $H_1(L_2^{6_1};\Z) \cong \Z_9$, so we also take a non-trivial character $\chi \colon \Z_9 \to \Z_3$.  Keeping track of the curve $\g$, one then needs our character on $H_1(L_2^{6_1};\Z)$ to map the image of $\g$ to a non-zero element of $\Z_3$.
    One way to check that this is possible would be to use our methods of the previous section; we have not done the calculation.
    Litherland's formula \cite{Litherland84} for Casson-Gordon obstructions would then show that they do not vanish, thus giving another proof that $S$ is not concordant to $H$.
\en

In our obstruction theory, we make use of the fact that $\HH$ has rank greater than one when we take finite covers using maps $\vph \colon H_1(X_L) \to A$ which are non-trivial on the meridians of both link components.  In order to make use of the full power of our obstruction theory, we need an example which also makes use of the fact that $\HH$ has rank greater than one in the non-triviality of the Witt class $\tau(S,\chi)$ i.e. the fact that we define an obstruction in $L^0(\C(\HH))$ and not just in $L^0(\C(t))$.  Such examples do not arise from satellite constructions.

Added in proof: On the other hand, as mentioned in the introduction, Min Hoon Kim informs us that he has verified the conjecture that the obstructions of this paper vanishes for links which are height $(3.5)$ Whitney tower/Grope concordant to the Hopf link.  Therefore our satellite links are not height $(3.5)$ concordant to the Hopf link, but the blow down technique cannot be used to prove this.

\section*{Appendix: Homology of finite abelian covers}

It is well--known that the order of the homology of a cyclic cover of a knot complement can be computed from the Alexander polynomial.
More precisely, let $K$ be a knot, let $\Delta_K(t)\in \zt$ be its Alexander polynomial,  let $n\in \N$ and denote by $X_K^n$ the $n$-fold cyclic cover of $X_K$.
Then Fox \cite{FoxIII} (see also \cite{We79} and \cite[Section~1.9]{Tu86}) showed that
 $H_1(X_K^n)\cong \Z\oplus G$, where $G$ is a group which satisfies
\begin{equation} \label{equ:h1k} |G|=\prod\limits_{k=0}^{n-1}\Delta_K(e^{2\pi ik/n}).\end{equation}
Here we denote by $|G|$ the order of $G$. We use the convention that  $|G|=0$ if  $G$ is infinite.

In general a na\"{\i}ve extension of the above formula to the case of finite abelian covers of links does not hold.
We refer to \cite{MM82}, \cite{Sa95}, \cite{HS97} and \cite{Po04} for results relating the size of the homology of finite abelian (branched) covers of link complements to evaluations of the multivariable
Alexander polynomial and its sublinks.

In this paper we are only interested in 2--component links with linking number one. One of the guiding themes of this paper is that such links share many formal properties with knots, in particular many statements about abelian invariants of knots can be generalised fairly easily to such links.  In this appendix we will give an elementary proof that the analogue of (\ref{equ:h1k}) holds in our case.

Before we state our theorem we have to fix some notation.
Let $L$ be an oriented 2-component link. We  denote by $s$ and $t$ the canonical elements in $H_1(X_L;\Z)$ corresponding to the meridians.
If $\varphi\colon H_1(X_L)\to A$ is a homomorphism and $\eta\colon A\to S^1$ a character, then we can consider  $\eta(\varphi(s))\in S^1$ and $\eta(\varphi(t))\in S^1$.
We now have the following theorem.

\begin{theorem*}
Let $L$ be an oriented 2-component link with linking number one. Let $\varphi\colon H_1(X_L)\to A$ be an epimorphism onto a finite abelian group.
Then $H_1(X_L^\varphi)\cong \Z^2\oplus G$, where $G$ is a group which satisfies
\[ |G|=\prod\limits_{\eta\in \operatorname{Hom}(A,S^1)}\Delta_L(\eta(\varphi(s)),\eta(\varphi(t))).\]
Here $\Delta_L(s,t)$ denotes the Alexander polynomial of $L$.
\end{theorem*}

\begin{proof}
We can view $X_L$ as a finite CW complex. We denote by $Y$ one of the two boundary components of $X_L$.
We can collapse $X_L$ onto a 2-complex $X$ which contains $Y$ and such that all 0-cells lie in $Y$. We denote by $r$ the number of 1--cells in $X\sm Y$.
Note that the Euler characteristic of $X$ and $Y$ is zero; it follows that the number of 2--cells in $X\sm Y$ also equals $r$.
Furthermore note that the inclusions also induce isomorphisms $H_1(Y;\Z)\to H_1(X;\Z)\to H_1(X_L;\Z)$. Also note that $H_1(Y;\zst)=0$.  It follows from the long exact sequence of $\zst$-homology groups of the pair $(X,Y)$, that
\[ H_1(X,Y;\zst)\cong H_1(X;\zst), \]
and the latter is of course isomorphic to $H_1(X_L;\zst)$. Since $X$ is a 2-complex and since $Y$ contains all 0-cells of $X$ we obtain the following exact sequence of $\zst$-modules:
\[ \ba{ccccccc} & 0&\to& H_2(X,Y;\zst)&\to \\
\to &C_2(X,Y;\zst)&\to& C_1(X,Y;\zst)&\to \\
\to &H_1(X,Y;\zst)&\to& 0.\ea \]
Note that the boundary map
\[  C_2(X,Y;\zst)\to C_1(X,Y;\zst)\] is represented by an $r\times r$--matrix $P$.
Furthermore note that by definition of the Alexander polynomial, and since $H_1(X,Y;\zst)\cong H_1(X_L;\zst)$ we have $\det(P)=\Delta_L$.

Now let $\varphi\colon H_1(X_L;\Z)\to A$ be an epimorphism onto a finite abelian group.  We write $k=|A|$.
We denote the resulting epimorphism $H_1(X;\Z)\cong H_1(X_L;\Z)\xrightarrow{\varphi} A$ by $\varphi$ as well.
We also denote by $\varphi$ the induced map $\zst=\Z[H_1(X_L;\Z)]\to \Z[A]$ and finally we denote by $\varphi(P)$ the $r\times r$--matrix given by applying $\varphi$ to each entry of $P$.
We then obtain the following exact sequence of $\Z[A]$-modules:
\[ \ba{ccccccc}     &&   0&\to& H_2(X,Y;\Z[A])&\to \\
&\to& C_2(X,Y;\Z[A])&\xrightarrow{\partial} &C_1(X,Y;\Z[A])&\to \\
&\to& H_1(X,Y;\Z[A])&\to &0,\ea \]
where the boundary  map is presented by $\varphi(P)$.

We now pick an isomorphism $\Z[A]\to \Z^{k}$ of free $\Z$-modules. Note that the action of $H_1(X_L;\Z)$ on $\Z[A]$ given by $\varphi$ and by left multiplication of $A$ on $\Z[A]$ gives rise to a representation
$H_1(X_L;\Z)\to \operatorname{GL}(k,\Z)$, which we also denote by $\varphi$. Note that this map extends to a ring homomorphism $\zst\to \Z[H_1(X_L;\Z)]\to M(k,\Z)$ which we also denote by $\varphi$.
Finally we denote by $\varphi(P)$ the $kr\times kr$-matrix over $\Z$ given by applying the above ring homomorphism to each entry of $\varphi(P)$.
It follows from the above, that
\[ |H_1(X,Y;\Z[A])|=\det(\varphi(P)).\]
It is well--known that for any finite abelian group $B$ the regular representation $B\to \aut(\C[B])=\operatorname{GL}(|B|,\C)$ is conjugate to the representation
which is given by the direct sum of all characters $\eta\colon B\to S^1$.
It now follows that
\[ \ba{rcl} |H_1(X,Y;\Z[A])|&=&\det(\varphi(P))\\
&=&\prod_{\eta\in \operatorname{Hom}(A,S^1)}\det(\eta(P))\\
&=&\prod_{\eta\in \operatorname{Hom}(A,S^1)}\eta(\det(P))\\
&=&\prod_{\eta\in \operatorname{Hom}(A,S^1)}\eta(\Delta_L(s,t))\\
&=&\prod_{\eta\in \operatorname{Hom}(A,S^1)}\Delta_L(\eta(s),\eta(t)).\ea \]

Now note that we have an exact sequence
\[ \ba{ccclccccccc} &&&&& H_2(X,Y;\Z[A])&\to\\
\to & H_1(Y;\Z[A])&\to& H_1(X;\Z[A])&\to& H_1(X,Y;\Z[A])&\to\\
\to &H_0(Y;\Z[A])&\to& H_0(X;\Z[A]).\ea \]
Note that the last two groups are isomorphic to $\Z$ and the last map is an isomorphism.
We now consider the following commutative diagram
\[ \xymatrix{ H_1(Y;\Z[A])\ar[d]\ar[r] & H_1(X;\Z[A]) \ar[d] \\ H_1(Y)\ar[r] &H_1(X;\Z),} \]
where the vertical maps are induced by the surjection from the $\varphi$-cover to the base space. Note that the left hand vertical map is injective since $Y$ is a torus,
and the bottom horizontal map is injective since $L$ is a 2-component link with linking number one. It follows that the top horizontal map is also injective.
Finally  note that $H_1(Y;\Z[A])\cong \Z^2$.
Combining all of the above we obtain a short exact sequence
\[ 0\to \Z^2\to H_1(X;\Z[A])\to H_1(X,Y;\Z[A])\to 0.\]
The above calculation of $|H_1(X,Y;\Z[A])|$ implies the theorem.
\end{proof}

\subsection*{Acknowledgment.}

This research was mostly conducted while MP was a visitor at the Max Planck Institute for Mathematics in Bonn, which he would like to thank for its generous hospitality.

We furthermore would like to thank Jae Choon Cha, Jim Davis, Jonathan Hillman, Kent Orr, Danny Ruberman, Dan Silver and Lorenzo Traldi for helpful conversations, comments and feedback.

\bibliographystyle{annotate}
\bibliography{markbib1}

\begin{thebibliography}{CKRS10}

\bibitem[CF10]{cha_twisted_2010}
J.~C. Cha and S.~Friedl.
\newblock Twisted torsion invariants and link concordance.
\newblock {\em To appear in Forum Mathematicum, Preprint:
  http://arxiv.org/abs/1001.0926}, 2010.


\bibitem[CG86]{CassonGordon}
{A}. {C}asson and {C}.~{M}c{A}. {G}ordon.
\newblock {C}obordism of classical knots.
\newblock In {\em A la {R}echerche de la {T}opologie {P}erdue}, volume~62 of
  {\em Progr. Math.}, pages 181--199. Birkhauser Boston, 1986.

\bibitem[Cha07]{Ch07}
J.~C. Cha.
\newblock The structure of the rational concordance group of knots.
\newblock {\em Mem. Amer. Math. Soc.}, 189(885):x+95, 2007.


\bibitem[Cha10]{cha_Hirzebuch_type_defects}
J.~C. Cha.
\newblock Link concordance, homology cobordism, and {H}irzebruch-type defects
  from iterated $p$-covers.
\newblock {\em Journal of the European Mathematical Society}, 12:555 -- 610,
  2010.


\bibitem[Cha12]{Cha_symmetric_Whitney_towers}
J.~C. Cha.
\newblock Symmetric {W}hitney tower cobordism for bordered $3$--manifolds and
  links.
\newblock {\em ar{X}iv:1204.4968 [math.GT]}, 2012.


\bibitem[CK99a]{CK99b}
J.~C. Cha and K.~H. Ko.
\newblock On equivariant slice knots.
\newblock {\em Proc. Amer. Math. Soc.}, 127(7):2175--2182, 1999.


\bibitem[CK99b]{CK99}
J.~C. Cha and K.~H. Ko.
\newblock Signature invariants of links from irregular covers and non-abelian
  covers.
\newblock {\em Math. Proc. Cambridge Philos. Soc.}, 127(1):67--81, 1999.


\bibitem[CK02]{CK02}
J.~C. Cha and K.~H. Ko.
\newblock Signatures of links in rational homology spheres.
\newblock {\em Topology}, 41:1161--1182, 2002.


\bibitem[CK06]{CK06}
J.~C. Cha and K.~H. Ko.
\newblock Signature invariants of covering links.
\newblock {\em Trans. Amer. Math. Soc.}, 358:3399--3412, 2006.


\bibitem[CK08]{CK08}
J.~C. Cha and T.~Kim.
\newblock Covering link calculus and iterated {B}ing doubles.
\newblock {\em Geometry and Topology}, 12:2172--2201, 2008.


\bibitem[CKRS10]{CKRS10}
J.C. Cha, T.~Kim, D.~Ruberman, and S.~Strle.
\newblock Smooth concordance of links topologically concordant to the {H}opf
  link.
\newblock {\em arXiv:1012.2045v1[math.GT]}, 2010.


\bibitem[CLR08]{CLR08}
J.~C. Cha, C.~Livingston, and D.~Ruberman.
\newblock Algebraic and {H}eegaard-{F}loer invariants of knots with slice
  {B}ing doubles.
\newblock {\em Math. Proc. Camb. Phil. Soc.}, 144:403--410, 2008.


\bibitem[CO90]{CO90}
T.D. Cochran and K.E. Orr.
\newblock Not all links are concordant to boundary links.
\newblock {\em Bulletin of the A.M.S.}, 23(1):99--106, 1990.


\bibitem[CO93]{CO93}
T.D. Cochran and K.E. Orr.
\newblock Not all links are concordant to boundary links.
\newblock {\em Annals of Math.}, 138:519--554, 1993.


\bibitem[CO09]{chaorr}
J.~C. Cha and K.~E. Orr.
\newblock ${L}^{(2)}$-signatures, homology localization and amenable groups.
\newblock {\em To appear in Communications on Pure and Applied Math., Preprint:
  arXiv:09103.3700}, 2009.


\bibitem[COT03]{COT}
T.~D. Cochran, K.~E. Orr, and P.~Teichner.
\newblock {K}not concordance, {W}hitney towers and {$L^{(2)}$} signatures.
\newblock {\em Ann. of Math. (2)}, 157, no. 2:433--519, 2003.


\bibitem[CST12]{CST11}
J.~Conant, R.~Schneiderman, and P.~Teichner.
\newblock {W}hitney tower concordance of classical links.
\newblock {\em Geometry and Topology}, 16:1419--1479, 2012.


\bibitem[Dav06]{Da06}
J.~Davis.
\newblock A two component link with {A}lexander polynomial one is concordant to
  the {H}opf link.
\newblock {\em Math. Proc. Cambridge Philos. Soc.}, (140, no. 2):265--268,
  2006.


\bibitem[Fox56]{FoxIII}
R.~H. Fox.
\newblock Free differential calculus. {III}. {S}ubgroups.
\newblock {\em Ann. of Math. (2)}, 64:407--419, 1956.


\bibitem[FP12]{FP10}
S.~Friedl and M.~Powell.
\newblock An injectivity theorem for {C}asson--{G}ordon type representations
  relating to the concordance of knots and links.
\newblock {\em Bull. Korean Math. Soc.}, 49:395--409, 2012.


\bibitem[FQ90]{FQ}
M.~Freedman and F.~Quinn.
\newblock {\em Topology of Four Manifolds}.
\newblock Princeton Univ. Press, 1990.


\bibitem[Fre84]{Fre84}
M.~H. Freedman.
\newblock The disk theorem for four-dimensional manifolds.
\newblock {\em Proc. Internat. Congr. Math. vol. 1, 2 (Warsaw, 1983), PWN},
  pages 647--663, 1984.


\bibitem[Fri03]{Friedl}
S.~Friedl.
\newblock {\em Eta Invariants as Sliceness Obstructions and their Relation to
  {C}asson-{G}ordon Invariants}.
\newblock PhD Thesis, Brandeis University, 2003.


\bibitem[Fri04]{Fr04}
S.~Friedl.
\newblock Eta invariants as sliceness obstructions and their relation to
  {C}asson-{G}ordon invariants.
\newblock {\em Algebraic and Geometric Topology}, 4:893--934, 2004.


\bibitem[Fri05]{Fri05}
S.~Friedl.
\newblock Link concordance, boundary link concordance and eta-invariants.
\newblock {\em Math. Proc. Cambridge Philos. Soc.}, (138, no. 3):437--460,
  2005.


\bibitem[Har08]{Ha08}
S.~Harvey.
\newblock Homology cobordism invariants and the {C}ochran-{O}rr-{T}eichner
  filtration of the link concordance group.
\newblock {\em Geom. Topol.}, (12, no.1):387--430, 2008.


\bibitem[Hil02]{Hi02}
J.~A. Hillman.
\newblock {\em Algebraic Invariants of Links}.
\newblock Series on Knots and Everything 32 (World Scientific Publishing Co.),
  2002.


\bibitem[HS97]{HS97}
J.~A. Hillman and M.~Sakuma.
\newblock On the homology of finite abelian coverings of links.
\newblock {\em Canad. Math. Bull.}, 40:309--315, 1997.


\bibitem[Kaw77]{Ka77}
A.~Kawauchi.
\newblock On quadratic forms of 3-manifolds.
\newblock {\em Invent. Math.}, (43):177--198, 1977.


\bibitem[Kaw78]{Ka78}
A.~Kawauchi.
\newblock On the {A}lexander polynomials of cobordant links.
\newblock {\em Osaka J. Math.}, (15, no. 1):151--159, 1978.


\bibitem[Let00]{Let00}
C.~F. Letsche.
\newblock An obstruction to slicing knots using the eta invariant.
\newblock {\em Math. Proc. Cambridge Phil. Soc.}, 128(2):301--319, 2000.


\bibitem[Lev69]{Levine}
J.~Levine.
\newblock Knot cobordism groups in codimension two.
\newblock {\em Comment. Math. Helv.}, 44:229--244, 1969.


\bibitem[Lev94]{Le94}
J.~Levine.
\newblock Link invariants via the eta invariant.
\newblock {\em Comm. Math. Helv.}, 69:82--119, 1994.


\bibitem[Lev07]{Le07}
J.~Levine.
\newblock Concordance of boundary links.
\newblock {\em Journal of Knot Theory and its Ramifications}, (16, no.
  9):1111--1120, 2007.


\bibitem[Lit84]{Litherland84}
R.~A. Litherland.
\newblock Cobordism of satellite knots.
\newblock In {\em Four-manifold theory ({D}urham, {N}.{H}., 1982)}, volume~35
  of {\em Contemp. Math.}, pages 327--362. Amer. Math. Soc., 1984.

\bibitem[Mil57]{Mil57}
J.~W. Milnor.
\newblock Isotopy of links.
\newblock {\em Alg. geom. and top. {A} symposium in honor of {S}. {L}efschetz},
  pages 280--306, 1957.


\bibitem[MM82]{MM82}
J.~P. Mayberry and K.~Murasugi.
\newblock {T}orsion-groups of abelian coverings of links.
\newblock {\em Trans. Amer. Math. Soc.}, 271:143--173, 1982.


\bibitem[Mur67]{Mu67}
K.~Murasugi.
\newblock On a certain numerical invariant of link types.
\newblock {\em Trans. Amer. Math. Soc.}, (117):387--422, 1967.


\bibitem[Nak78]{Na78}
Y.~Nakagawa.
\newblock On the {A}lexander polynomials of slice links.
\newblock {\em Osaka J. Math.}, (15, no. 1):161--182, 1978.


\bibitem[Por04]{Po04}
J.~Porti.
\newblock {M}ayberry-{M}urasugi's formula for links in homology 3-spheres.
\newblock {\em Proc. Amer. Math. Soc.}, 132:3423--3431, 2004.


\bibitem[Sak95]{Sa95}
M.~Sakuma.
\newblock {H}omology of abelian coverings of links and spatial graphs.
\newblock {\em Canad. J. Math.}, 47:201--224, 1995.


\bibitem[Smo89]{Sm89}
L.~Smolinsky.
\newblock Invariants of link cobordism.
\newblock {\em Proceedings of the 1987 Georgia Topology Conference (Athens, GA,
  1987). Topology Appl.}, 32:161--168, 1989.


\bibitem[Sta65]{St65}
J.~Stallings.
\newblock Homology and central series of groups.
\newblock {\em Journal of Algebra}, 2:170--181, 1965.


\bibitem[Tri69]{Tristram69}
A.~G. Tristram.
\newblock Some cobordism invariants for links.
\newblock {\em Proc. Cambridge Philos. Soc.}, 66:251--264, 1969.


\bibitem[Tur86]{Tu86}
V.~Turaev.
\newblock {R}eidemeister torsion in knot theory.
\newblock {\em Russian Math. Surveys}, (41):119--182, 1986.


\bibitem[Web79]{We79}
C.~Weber.
\newblock Sur une formule de {R}. {H}. {F}ox concernant l'homologie des
  rev\^{e}tements cycliques.
\newblock {\em Enseign. Math. (2)}, 25:261--272, 1979.


\end{thebibliography}

\end{document}